\documentclass[letterpaper,10pt]{amsart}

\usepackage[all]{xy}                        %

\CompileMatrices                            % Faster

\UseTips                                    % Use

\input xypic
\usepackage[bookmarks=true]{hyperref}       % Hyperref
\usepackage{amssymb,latexsym,amsmath,amscd}
\usepackage{xspace}

\usepackage{graphicx}
\usepackage{color}
\reversemarginpar

\theoremstyle{plain}
\newtheorem{theorem}{Theorem}[section]
\newtheorem*{theorem*}{Theorem}
\newtheorem{proposition}[theorem]{Proposition}
\newtheorem{corollary}[theorem]{Corollary}
\newtheorem{lemma}[theorem]{Lemma}

\theoremstyle{definition}
\newtheorem{definition}[theorem]{Definition}

\newtheorem{remark}[theorem]{Remark}

\newcommand{\enm}[1]{\ensuremath{#1}}          %
\newcommand{\op}[1]{\operatorname{#1}}
\newcommand{\cal}[1]{\mathcal{#1}}

\newcommand{\CC}{\enm{\mathbb{C}}}
\newcommand{\II}{\enm{\mathbb{I}}}
\newcommand{\NN}{\enm{\mathbb{N}}}

\newcommand{\QQ}{\enm{\mathbb{Q}}}
\newcommand{\GG}{\enm{\mathbb{G}}}
\newcommand{\ZZ}{\enm{\mathbb{Z}}}

\renewcommand{\AA}{\enm{\mathbb{A}}}
\renewcommand{\II}{\enm{\mathbb{I}}}
\newcommand{\PP}{\enm{\mathbb{P}}}

\newcommand{\Aa}{\enm{\cal{A}}}

\newcommand{\Cc}{\enm{\cal{C}}}
\newcommand{\Dd}{\enm{\cal{D}}}
\newcommand{\Ee}{\enm{\cal{E}}}
\newcommand{\Ff}{\enm{\cal{F}}}
\newcommand{\Gg}{\enm{\cal{G}}}
\newcommand{\Hh}{\enm{\cal{H}}}
\newcommand{\Ii}{\enm{\cal{I}}}

\newcommand{\Ll}{\enm{\cal{L}}}

\newcommand{\Oo}{\enm{\cal{O}}}

\newcommand{\Tt}{\enm{\cal{T}}}

\renewcommand{\phi}{\varphi}
\renewcommand{\theta}{\vartheta}
\renewcommand{\epsilon}{\varepsilon}

\newcommand{\Hom}{\op{Hom}}
\newcommand{\Ext}{\op{Ext}}

\newcommand{\Supp}{\op{Supp}}

\renewcommand{\to}[1][]{\xrightarrow{\ #1\ }}

\newcommand{\old}[1]{}

\begin{document}

\title[Double lines on quadric hypersurfaces]{Double lines on quadric hypersurfaces}

\author{Edoardo Ballico and Sukmoon Huh}

\address{Universit\`a di Trento, 38123 Povo (TN), Italy}
\email{edoardo.ballico@unitn.it}

\address{Sungkyunkwan University, Suwon 440-746, Korea}
\email{sukmoonh@skku.edu}

\thanks{The first author is partially supported by MIUR and GNSAGA of INDAM (Italy). }


\begin{abstract}
We study double line structures in projective spaces and quadric hypersurfaces, and investigate the geometry of irreducible components of Hilbert scheme of curves and moduli of stable sheaves of pure dimension $1$ on a smooth quadric threefold.  
\end{abstract}

\maketitle

\tableofcontents

\section{Introduction}
Let $X$ be a smooth projective variety over $\CC$, the field of complex numbers. For a fixed polynomial $\chi(t)$, C. Simpson introduced in \cite{Simp} the coarse projective moduli spaces $\mathbf{M}_{\chi(t)}(X)$ of semistable sheaves with the Hilbert polynomial $\chi(t)$. If the degree $\chi(t)$ is $d$, then the support of each sheaf in $\mathbf{M}_{\chi(t)}(X)$ is $d$-dimensional subvariety of $X$. The study on these moduli spaces give inspiration on the study of Hilbert schemes $\mathbf{Hilb}_{\chi(t)}(X)$ of curves on $X$ with the Hilbert polynomial $\chi(t)$, because certain components of $\mathbf{M}_{\chi(t)}(X)$ can be viewed as compactifications of an open part of the corresponding $\mathbf{Hilb}_{\chi(t)}(X)$. There have been several investigations on the relation of these spaces in the case of projective spaces $X=\PP^n$, e.g. \cite{CCM, FT,LP}. 

In this paper we first investigate rational ribbons, i.e. double structures on $\PP^1$. Rational ribbons and their canonical embeddings were studied in \cite{be} and we show several facts on the Hilbert scheme of double lines in projective spaces and $n$-dimensional quadric hypersurfaces $Q_n$. Note that the study on the families of double lines in projective spaces is done in \cite{NNS}. Then we describe  birational properties of $\mathbf{M}_{\chi(t)}(Q_3)$ and $\mathbf{Hilb}_{\chi(t)}(Q_3)$ when the leading coefficient of a linear polynomial $\chi(t)$ is two. A double line in $Q_n$ is a locally Cohen-Macaulay curve with the Hilbert polynomial $\chi(t)=2t+a$ for some constant $a$ and  a line as its reduction. The information on the set of double lines turns out to be very valuable in describing Simpson's moduli spaces and Hilbert schemes. Letting $\Dd_a$ be the subscheme consisting of double lines in $Q_3$ with the Euler characteristic $a$, we observe that $\Dd_1$ consists of reducible and non-reduced conics on $Q_3$. Indeed $\mathbf{Hilb}_{2t+1}(Q_3)$ consists of the conics on $Q_3\subset \PP^4$ and we observe the following:

\begin{theorem}
We have $\mathbf{Hilb}_{2t+1}(Q_3) \cong \mathbf{M}_{2t+1}(Q_3) \cong \mathrm{Gr}(3,5)$, the Grassmannian variety parametrizing projective planes in $\PP^4$.
\end{theorem}
We also follow the notion of $\alpha$-stable pairs as in \cite{He} to consider its moduli spaces $\mathbf{M}^{\alpha}_{\chi(t)}(Q_3)$, and observe that they are all isomorphic to $\mathrm{Gr}(3,5)$ if $\chi(t)=2t+1$. In other words, there is no wall-crossing. Then we turn our interest to the case of $\chi(t)=2t+2$, when there is again no wall-crossing. Our main result is as follows:

\begin{theorem}
For the Hilbert polynomial $\chi(t)=2t+2$, we have the following description on the three moduli spaces:
\begin{enumerate}
\item $\mathbf{Hilb}_{\chi(t)}(Q_3)$ consists of two rational irreducible components, $\Hh_1$ and $\Hh_2$, of dimension $9$ and $6$ respectively, and it is smooth outside $\Hh_1 \cap \Hh_2$.
\item $\mathbf{M}_{\chi(t)}(Q_3)$ consists of two irreducible components, $\mathfrak{M}_1$ and $\mathfrak{M}_2$, of dimension $6$ both, and $\mathfrak{M}_1$ is rational and smooth outside $\mathfrak{M}_1 \cap \mathfrak{M}_2$. 
\item $\mathbf{M}^{\alpha}_{\chi(t)}(Q_3)$ consists of two rational irreducible components, $\mathfrak{N}_1$ and $\mathfrak{N}_2$, of dimension $7$ and $6$ respectively, and it is smooth outside $\mathfrak{N}_1\cap \mathfrak{N}_2$. 
\end{enumerate}
\end{theorem}
Note that we have no geometric description on $\mathfrak{M}_2$ because it consists only of strictly semistable sheaves. Moreover we give full description of elements in each irreducible components and intersections:
\begin{itemize}
\item $\Hh_1$ consists of non-locally Cohen-Macaulay curves and $\Hh_2$ is the closure of locally Cohen-Macaulay curves. Their intersection consists of singular conics $D$ with an extra point $p\in D_{\mathrm{sing}}$ such that the hyperplane section containing the curve is singular at $p$. 
\item $\mathfrak{M}_{1,\mathrm{red}}$ is parametrized by the space of conics in $Q_3$ and $(\mathfrak{M}_1 \cap \mathfrak{M}_2)_{\mathrm{red}}$ is parametrized by the space of singular conics in $Q_3$. $\mathfrak{M}_2$ has a one-to-one correspondence to $\mathrm{Sym}^2(\PP^3)$, the set of pairs of two lines in $Q_3$. 
\item $\mathfrak{N_1}$ is birational to the incidence variety of the space of conics in $Q_3$ and $\mathfrak{N}_2$ is birational to $\mathrm{Sym}^2(\PP^3)$. 
\end{itemize}

Set-theoretic description of the component induced from locally Cohen-Macaulay curves is relatively easy and the main ingredients in the study of the other component are the family of double lines $\Dd_1$ and $\Dd_2$. We classify non-locally Cohen-Macaulay curves with respect to the hyperplane section containing them and study their corresponding pure sheaves with the deformation data. 

Let us summarize here the structure of this paper. In section $2$, we introduce the definitions and main properties that will be used throughout the paper, mainly those related to stability and $\alpha$-stability conditions. In section $3$, we pay attention to the Hilbert schemes of double lines in projective spaces and quadric hypersurfaces to conclude their irreducibility in some cases. In section $4$, as a warm-up case, we describe the moduli spaces with the Hilbert polynomial $\chi(t)=2t+1$. Finally in section $5$, we deal with the case of $\chi(t)=2t+2$ and describe the irreducible components of each moduli spaces and their intersections.

%%%%%%%%%%%%%%%%%%%%%%%%%%%%%%

\section{Definitions and preliminaries}

Let $Q_n$ be a smooth quadric hypersurface of the complex projective space $\PP^{n+1}$. Then we have 
$$\mathrm{Pic}(Q_n)=H^2(Q_n, \ZZ)=\ZZ \langle h \rangle ~~~~~~\text{ for }n\ge 3$$
where $h$ is the class of a hyperplane section. If $n=3$, the cohomology ring $H^*(Q_3, \ZZ)$ is generated by $h$, a line $l\in H^4(Q_3, \ZZ)$ and a point $p\in H^6(Q_3, \ZZ)$ with the relations: $h^2=2l, h \cdot l=p$, $h^3=2p$. If there is no confusion, we will denote $Q_3$ simply by $Q$.

\begin{definition}
Let $\Ff$ be a pure sheaf of dimension $1$ on $Q$ with the Hilbert polynomial $\chi_{\Ff}(t)=\mu t+\chi$ with respect to $\Oo_{Q}(1)$. The {\it p-slope} of $\Ff$ is defined to be $p(\Ff)=\chi/\mu$. $\Ff$ is called {\it semistable (stable)} if 
\begin{enumerate}
\item $\Ff$ does not have any $0$-dimensional torsion, and 
\item for any proper subsheaf $\Ff '$, we have
$$p(\Ff ')=\frac{\chi'}{\mu '} \le (<) \frac{\chi}{\mu}=p(\Ff)$$
where $\chi_{\Ff '}(t)=\mu '+\chi'$. 
\end{enumerate}
\end{definition}

For every semistable $1$-dimensional sheaf $\Ff$ with $\chi_{\Ff}(t)=\mu t+\chi$, let us define $C_{\Ff }:=\Supp (\Ff)$ to be its scheme-theoretic support and then it corresponds to $\mu l \in H^4(X)$. We often use slope stability (resp. slope semistability) instead of Gieseker stability (resp. semistability) with respect to $L:=\Oo_{Q}(1)$, just to simplify the notation; they should be the same, because the support is $1$-dimensional and so the inequalities for Gieseker and slopes $\chi /\mu$ are the same.

\begin{definition}
Let $\mathbf{M}(\mu,\chi)$ be the moduli space of semistable sheaves on $Q$ with linear Hilbert polynomial $\chi(t)=\mu t+\chi$. 
\end{definition}

Note that $\chi_{\Ff(a)}(t)=\chi_{\Ff}+\mu \cdot a$ and so we may assume $0<\chi \le \mu$. Our main interest in this article is on the case $\mu=2$; $\mathbf{M}(2,1)$ in Section $4$ and $\mathbf{M}(2,2)$ in Section $5$.

For a smooth projective variety $X\subset \PP^r$, let $\mathbf{Hilb}_X(\mu , \chi)$ be the Hilbert scheme of curves in $X$ with the Hilbert polynomial $\mu t+ \chi$. 

\begin{definition}
A {\it locally Cohen-Macaulay (for short, locally CM)} curves in $X$ is a $1$-dimensional subscheme $C\subset X$ whose irreducible components are all $1$-dimensional and that has no embedded points. We denote by $\mathbf{Hilb}_X(\mu, \chi)_+$ the subset of $\mathbf{Hilb}_X(\mu, \chi)$ parametrizing the locally CM curves with no isolated point. 
\end{definition}

If there is no confusion, we will simply denote $\mathbf{Hilb}_Q(\mu, \chi)$ by $\mathbf{Hilb}(\mu, \chi)$. 

\begin{remark}
For $\Ff \in \mathbf{M}(1,1)$, we have $\Ff \cong \Oo_L$ with a line $L\subset Q$. Note that we have ${TQ}_{|L} \cong \Oo_L(2)\oplus \Oo_L(1)\oplus \Oo_L$ and so $N_{L|Q}\cong \Oo_L(1)\oplus \Oo_L$. It implies $h^0(N_{L|Q})=3$ and $h^1(N_{C|Q})=0$. Thus $\mathbf{Hilb}(1,1)$ is smooth and of dimension $3$. It is well known that the family of lines in $Q$ is isomorphic to $\PP^3$. Hence we have $\mathbf{M}(1,1) \cong \PP^3$. 
\end{remark}

For a positive rational number $\alpha \in \QQ_{>0}$, a pair $(s,\Ff)$ of a purely $1$-dimensional sheaf $\Ff$ with $\chi_{\Ff}(t)=\mu t+\chi$ and a non-zero section $s:\Oo_{Q} \to \Ff$ is called $\alpha$-{\it semistable} if $\Ff$ is pure and for any non-zero proper subsheaf $\Ff' \subset \Ff$ with $\chi_{\Ff '}(t)=\mu't+\chi'$, we have
$$\frac{\chi '+\delta \cdot \alpha}{\mu '} \leq \frac{\chi+\alpha}{\mu}=:\mu_{\alpha}(s,\Ff),$$
where we take $\delta=1$ if the section $s$ factors through $\Ff'$ and $\delta=0$ if not. As usual, if the inequality is strict, we call it $\alpha$-{\it stable}. By \cite[Theorem 4.2]{He} the wall happens at $\alpha$ with which the strictly $\alpha$-semistability occurs. As a routine, we will write $(1,\Ff)$ for the pair of a sheaf with $\Ff$ with a non-zero section and $(0,\Ff)$ for the pair of sheaf with zero section . 

Let us denote by $\mathbf{M}^{\alpha}(\mu , \chi)$ the moduli space of $\alpha$-semistable pairs. Note that there are only finitely many critical values $\{\alpha_1, \ldots, \alpha_s \}$ for $\alpha$-stability with $\alpha_1<\cdots<\alpha_s$ in a sense that any $\alpha\in (\alpha_i, \alpha_{i+1})$ gives the same moduli spaces of $\alpha$-stable pairs. Notice that if $\alpha<\alpha_1$, then $\alpha$-stability is equivalent to the Gieseker stability and so there exists a forgetful map 
$$\mathbf{M}^{0+}(\mu, \chi):=\mathbf{M}^{\alpha}(\mu, \chi) \to \mathbf{M}(\mu, \chi).$$
If $\alpha>\alpha_s$, then the cokernel of the pair $\Oo_{Q} \to \Ff$ is supported at a $0$-dimensional subscheme and so we get the moduli of PT stable pairs, which maps naturally to $\mathbf{Chow}(\mu, \chi):=\bigsqcup_k \mathbf{Hilb}(\mu, \chi -k)_+ \times \mathrm{Sym}^k Q$:
$$\mathbf{M}^{\infty}(\mu, \chi):=\mathbf{M}^{\alpha} (\mu, \chi) \to \mathbf{Chow}(\mu, \chi).$$

%%%%%%%%%%%%%%%%%%%%%%%%%%

\section{Double lines in projective spaces and hyperquadrics}

\begin{definition}
For each $a\in \ZZ$, define
$$\Dd_a:=\{ A \in \mathbf{Hilb}(2,a)_+~|~ A_{\mathrm{red}} \text{ is a line }\}.$$
For the moment we take $\Dd _a$ as a set. In each case it would be clear which scheme-structure is used on it. Since $Q$ is a smooth threefold, \cite[Remark 1.3]{bf} says that each $C\in \Dd _a$ is obtained by the Ferrand construction and in particular it is a ribbon in the sense of \cite{be} with a line as its support. For each line $L\subset Q$, we let 
$$\Dd_a(L):=\{ A \in \Dd_a ~|~ A_{\mathrm{red}}=L\}.$$
\end{definition}

For each positive integer $a\in \NN$, let us denote by $T[a]$ the unique ribbon $T$ on $\PP^1$ with $\chi (\Oo _{T[a]}) =a$ whose reduced scheme is a line as in \cite[Theorem 1.2]{be}. Then we have an exact sequence
\begin{equation}\label{eqb1}
0\to \Oo _{\PP^1}(a-2) \to \Oo _{T[a]}\to \Oo _{\PP^1}\to 0.
\end{equation}
Note that $T=T[a]$ is a split ribbon, i.e. the sequence (\ref{eqb1}) splits as a sequence of $\Oo _{\PP^1}$-sheaves and each line bundle $\Ll$ on $T$ is uniquely determined by the integer $\deg (\Ll _{|T_{\mathrm{red}}})$ by \cite[Proposition 4.1]{be}. We denote this line bundle by $\Oo _T(t)$, where $t= \deg (\Ll _{|T_{\mathrm{red}}})$. Using the sequence (\ref{eqb1}), we get $h^1(\Oo _T(t)) =0$ and $h^0(\Oo _T(t)) = 2t+a$ for $t\ge 0$. To collect several aspects on double lines, we will first consider double line structures as morphisms in a more general setting. 

\begin{definition}
For $a\ge 1$ and $n\ge 3$, we define
\begin{equation}
\Dd_{a,n}:=\{ \text{locally CM  curves $C$ in $Q_n$ with } \chi_{\Oo_C}(t)=2t+a\}.
\end{equation}
Similarly we define 
\begin{equation}
\Cc_{a,r}:=\{ \text{locally CM curves $C$ in $\PP^r$ with } \chi_{\Oo_C}(t)=2t+a\}.
\end{equation}
\end{definition}

\begin{remark}
Let $\mathfrak{D}_{a,n}$ be the family of all possible embeddings $T[a]\to Q_n$ such that $\Oo_{T[a]}(1)$ is its hyperplane line bundle. Then the elements in $\Dd_{a,n}$ can be obtained as images of the embedding in $\mathfrak{D}_{a,n}$. Thus we get $\dim \Dd_{a,n}=\dim \mathfrak{D}_{a,n}-\dim \mathrm{Aut}(T[a])$. Similarly we may define $\mathfrak{C}_{a,n}$ for $\Cc_{a,r}$ as a family of embeddings. 
\end{remark}

\begin{lemma}\label{u1}
Let $T=T[a]$ with $a>0$. For $t>0$, we have the following:
\begin{enumerate}
\item the line bundle $\Oo _T(t)$ is very ample for $t>0$ and 
\item a general $4$-dimensional linear subspace $V\subseteq H^0(\Oo _T(t))$ induces an embedding of $T$ into $\PP^3$.
\end{enumerate}
\end{lemma}

\begin{proof}
Set $D:= T_{\mathrm{red}}$. Each line bundle $\Oo _T(t)$ with $t>0$, is spanned by the sequence (\ref{eqb1}). The line bundle $\Oo _T(1)$ is ample, because $\Oo _T(1)$ generates $\mathrm{Pic}(T)$ and $h^0(\Oo _T(-c)) =0$ for $c\gg 0$. Thus $\Oo _T(t)$ is ample for each $t>0$. The morphism $f: T\to \PP ^m$ with $m:= h^0(\Oo _T(t))-1$, induces an embedding of $D$ and so it is sufficient to prove that $f$ is a local embedding. 

Fix a point $p\in D$ and a non-zero tangent vector $v$ of $T$ at $p$. If $v$ is tangent to $D$, then $f_{|v}$ is an embedding because $f_{|D}$ is an embedding. If $v$ is not tangent to $D$, then we may use $h^1(\Oo _{\PP^1}(a-2+t)) =0$ to see that $h^0(\Ii _v(t)) < h^0(\Ii _p(t))$. Since
this is true for all $p$ and $v$, so the injective morphism $f$ is a local embedding by \cite[Proposition II.7.3]{Hartshorne}. Hence $f$ is an embedding. Since $\dim (T) = 1$ and each Zariski tangent space of $T$ has dimension $2$, a general linear projection of $f(T)$ into $\PP^3$ is an embedding.
\end{proof}

\begin{remark}\label{u1.1}
By Lemma \ref{u1}, we get that $\Cc _{a,r} \ne \emptyset$ for each $a>0$ and $r \ge 3$. The elements of $\Cc_{a,r}$ are obtained in the following way: fix an integer $s$ such that $4\le s \le \min \{r+1,a+3\}$ and let $V\subseteq H^0(\Oo _{T[a]}(1))$ be an $s$-dimensional linear subspace spanning $\Oo _{T[a]}(1)$ and inducing an embedding of $T[a]$. Note that the set of all such subspaces is a non-empty open subset of the Grassmannian $\mathrm{Gr}(s,H^0(\Oo _{T[a]}(1)))$. Thus $\Cc _{a,r}$ is irreducible and its general element is obtained by taking $s = \min \{r+1,a+3\}$. If $a\ge 3$, we get 
$\mathbf{Hilb}_{Q_n}(2,a)_+ $ and $\mathbf{Hilb}_{\PP^{n+1}}(2,a)_+$ contain no reduced curve,
because every reduced curve $C \subset \PP^r$ of degree $2$ has $\chi (\Oo _C)\le 2$.
\end{remark}

By Remark \ref{u1.1}, $\Dd _{a,n}$ with $n\ge 7$ contains the embeddings of $T[a]$ into the maximal linear subspaces of $Q_n$. Later in Proposition \ref{u5}, we get a description of $\Dd_{a,4}$ with $a=1$ by considering $Q_4$ as the Grassmannian $\mathrm{Gr}(2,4)$. 

For fixed $a\in \ZZ$, we get the following lemma for $T=T[a]$ using the sequence (\ref{eqb1}). 

\begin{lemma}\label{u1.2}
The line bundle $\Oo_T(t)$ is spanned if and only if $t\ge 0$. We also have
$$h^0(\Oo_T(t))= \left\{
                                           \begin{array}{lll}
                                             0, & \hbox{if $t\le -a+1$;}\\
                                             t+a-1, & \hbox{if $-a+2\le t<0$;} \\                                          
                                             t+a, & \hbox{if $t\ge 0$,}
                                            \end{array}
                                         \right.$$
and $h^1(\Oo _T(t)) =0$ for all $t\ge -a+1$.  
\end{lemma}

\begin{lemma}\label{u4}
For $C \in \Cc _{a,r}$ with $r\ge 2$, we have the following: 
\begin{itemize}
\item[(i)] $C$ is contained in each quadric hypersurface whose singular locus contains $L:=C_{\mathrm{red}}$ and in particular, $h^0(\Ii _C(2)) \ge { r \choose 2}$.
\item[(ii)] The linear system $|\Ii _C(2)|$ has no base points outside $L$.
\end{itemize}
\end{lemma}

\begin{proof}
Let $B\subset \PP^r$ be a quadric hypersurface. We have $B\in |(\mathcal {I}_L)^2(2)|$ if and only if the singular locus of $B$ contains $L$ and the set of all such quadrics is a projective space of dimension $\binom{r}{2}-1$. Thus the lemma follows from the inclusion $(\mathcal {I}_L)^2\subset \Ii _C$. Fix a point $p\in \PP^r\setminus L$ and let $H\subset \PP^r$ be any hyperplane containing $L$ such that $p\notin H$. Since $2H\in |\Ii _C(2)|$, so $p$ is not a base point of $|\Ii _C(2)|$.
\end{proof}

For each integers $r\ge 3$ and $a\ge 1$, let us define
$$\Cc_{a,r}^+ := \{ C\in \Cc_{a,r}~|~ h^0(\Ii_C(1))=0 \text{ or }h^1(\Ii_C(1))=0\},$$
i.e. the linear span of $C$ in $\PP^r$ has dimension $\min \{r, h^0(\Oo _{T[a]}(1))-1\}$. Then $\Cc _{a,r}^+$ is a dense open subset of $\Cc _{a,r}$.

\begin{lemma}\label{u1.3}
For each integer $a\ge 3$, we have 
$\Cc _{a,a+1} \ne \emptyset$ 
and all elements of $\Cc _{a,a+1}^+$ are projectively equivalent.
\end{lemma}

\begin{proof}
Recall that $h^0(\Oo _{T[a]}(1)) =a+1$ and that $\Oo _{T[a]}(1)$ is very ample. Thus we get $\Cc _{a,a+1} \ne \emptyset$ and all elements of $\Cc _{a,a+1}^+$ are projectively equivalent; they correspond to a choice of a basis of $H^0(\Oo _{T[a]}(1))$. 
\end{proof}

For any $C\in \Cc _{a,r}$, we have $h^0(\Oo _C(2)) =a+4$. In particular we have $h^i(\Ii_C(2))>0$ for $i=0,1$, if $a+4 > \binom{r+2}{2} -\binom{r}{2} = r+1$ by Lemma \ref{u4}.

\begin{proposition}\label{u1.5}
For $C\in \Cc _{3,4}^+$, we have
\begin{enumerate}
\item $h^i(\Ii _C(t)) =0$ for all $i \ge 3$ and $t\ge -4$.   
\item $h^2(\Ii _C(t)) =0$ for all $t\ge -1$,
\item $h^1(\Ii_C(t))= \left\{
                                           \begin{array}{ll}
                                             0, & \hbox{if $t>0$;}\\
                                             2, & \hbox{if $t=0$,} 
                                           \end{array}
                                         \right.$                         
\item $h^0(\Ii _C(t)) = \binom{t+4}{4}-2t-3$ for all $t\ge 1$ and $C$ is contained in a smooth quadric hypersurface.                                             
\end{enumerate}
\end{proposition}

\begin{proof}
Since $\dim ({C})=1$, so we have $h^2(\Ii _C(t)) = h^1(\Oo _C(t)) =0$ for all $t\ge -1$. Similarly if $i \ge 3$ and $t\ge -4$, then we have $h^i(\Ii _C(t)) = h^i(\Oo _{\PP^4}(t)) =0$. Since $h^0(\Ii _C)=0$, we have $h^1(\Ii _C) = h^0(\Oo _C)-1 = 2$. Since $C\in \Cc _{3,4}^+$, so it is linearly normal, i.e. $h^i(\Ii_C(1)) =0$ for $i=0,1$. 

Since $h^0(\Oo _C(2)) =7$ and $h^0(\Oo _{\PP^4}(2)) =\binom{6}{2}=15$, we have $h^0(\Ii _C(2)) = 8+h^1(\Ii _C(2))$. Set $L:= C_{\mathrm{red}}$. By Lemma \ref{u4} and the Bertini theorem, a general element of $|\Ii _C(2)|$ has singular locus contained in $L$.

Fix a point $p\in \PP^1$. The choice of the point $p$, i.e. a degree $1$ effective divisor of $\PP^1$, induces a surjective morphism $f: T[3] \to T[2]$ inducing an exact sequence
\begin{equation}\label{equ1}
0 \to \Oo_{T[2]} \to \Oo_{T[3]} \to \CC _p  \to 0
\end{equation}
by \cite[Theorem 1.1]{eg}. From (\ref{equ1}) we get that the map $u: H^0(\Oo _{T[3]}(1))^\vee \to H^0(\Oo _{T[2]}(1))^\vee$ is surjective with a $1$-dimensional kernel. Since $C$ is linearly normal, $\mathrm{ker}(u)$ corresponds to a unique point $o\in \PP^4$. Since $f$ is a morphism, so we have $o\notin C$ and $f$ corresponds to the linear projection from $o$. Since $f$ is induced by the linear projection from $o$, the set $\Sigma _o$ of all quadric cones containing $C$ and with $o$ contained in their vertices has dimension $h^0(\PP^3,\Ii _{C'}(2)) -1=3$ with $C'=f(C)$. Since $C'$ is contained in a smooth quadric surface, we have $W=\{o\}$ for a general $W\in \Sigma _o$. In particular, $L$ is contained in the smooth locus of $W$ and so a general $A\in |\Ii _C(2)|$ is smooth at all points of $L$. Hence $A$ is smooth. The set of all $B\in |\Oo _{\PP^4}(2)|$ singular at $o$ has codimension $5$ in $|\Oo _{\PP^4}(2)|$. Since $\dim (\Sigma _o)= 3$ and $h^0(\Ii _C(2)) \ge 8$, we get $h^0(\Ii _C(2)) =8$ and so $h^1(\Ii _C(2)) =0$. Since $h^i(\Ii _C(3-i)) =0$ for $i=2,3,4$, so the Castelnuovo-Mumford lemma gives $h^1(\Ii _C(t)) =0$ for all $t\ge 3$.
\end{proof}

\begin{proposition}\label{u1.4}
We have $\mathbf{Hilb}(2,3)_+=\Dd _3 \ne \emptyset$.
\end{proposition}

\begin{proof}
Since $Q$ is a smooth three-fold, every degree $2$ locally Cohen-Macaulay double structure on a line is ribbon (\cite[Remark 1.3]{bf}).
No reduced curve $D\subset Q$ has $\chi (\Oo _D) =3$. Therefore $\mathbf{Hilb}(2,3)_+=\Dd _3$. We have $\Dd _3\ne \emptyset$, because every $C\in \Cc _{3,4}^+$ is contained in a smooth quadric hypersurface by Proposition \ref{u1.5}.
\end{proof}

\begin{lemma}\label{u2}
Let $\Ee$ be a vector bundle of rank $r\ge 2$ on $T=T[a]$. If $(a_1, \ldots, a_r) \in \ZZ^{\oplus r}$ is the splitting type of $\Ee _{|T_{\mathrm{red}}}$, i.e. $\Ee_{|T_{\mathrm{red}}} \cong \oplus _{i=1}^r\Oo_{T_{\mathrm{red}}}(a_i)$, then we have $\Ee \cong \oplus _{i=1}^{r} \Oo _T(a_i)$.
\end{lemma}

\begin{proof}
We use induction on $r$; the case $r=1$ being true by \cite[Proposition 4.1]{be}. Assume $r\ge1$ and that the lemma is true for lower ranks. 

Let $s$ be the maximal positive integer $i$ such that $a_i=a_1$ for all $i\le s$ and let $h$ be the minimal positive integer $i\le r$ such that $a-2+a_i -a_1\ge 0$. Tensoring (\ref{eqb1}) with $\Ee (-a_1)$, we get $h^0(\Ee (-a_1)) = s + \sum _{i=1}^{h} (a+a_i-a_1-1)$. Fix a general map $f: \Oo _X(a_1)\to \Ee$. Since we have $h^0(\Ee (-a_1)) > \sum _{i=1}^{h} (a+a_i-a_1-1)$, so $f$ induces an injective map $\phi: \Oo _{\PP^1}(-a_1)\to \oplus _{i=1}^{r} \Oo _{\PP^1}(a_i-a_1)$ with locally free cokernel. Thus $f$ is injective also with locally free cokernel, say $\Ff$. By the inductive assumption we have $\Ff \cong \oplus _{i=2}^{r} \Oo _T(a_i)$. Since $a$ is nonnegative, so the sequence (\ref{eqb1}) gives $h^1(\Oo _T(j)) =0$ for all $j\ge 0$. Hence every extension of $\Ff$ by $\Oo _T(a_1)$ splits and we get $\Ee \cong \oplus _{i=1}^{r} \Oo _T(a_i)$.
\end{proof}

\begin{definition}
The sequence of integers $(a_1,\dots ,a_r)$ with $a_1 \ge \cdots \ge a_r$ in the statement of Lemma \ref{u2} is called the {\emph {splitting type}} of $\Ee$. 
\end{definition}

\begin{remark}\label{u2.1}
By Lemma \ref{u1.2} the bundle $\Ee \cong \oplus _{i=1}^{r} \Oo _T(a_i)$ is spanned if and only
if $a_r\ge 0$. Since $\dim (T) =1$, so a dimension counting gives that if $\Ee$ is spanned and $h^0(\Ee )>r$, then $\Ee$ is spanned by a general $(r+1)$-dimensional linear subspace of $H^0(\Ee )$ by \cite[Theorem 2]{a}. Again by Lemma \ref{u1.2}, if $\Ee$ is spanned, then we get $h^1(\Ee )=0$. Now assume $a_r\ge 0$ and so $\Ee$ is spanned. By Lemma \ref{u1}, the pair $(\Ee ,H^0(\Ee ))$ induces an embedding of $T$ into the Grassmannian $\mathrm{Gr}(N,r)$ with $N = ra-r +2r(a_1+\cdots +a_r)$ if and only if $a_1>0$. In the case of $a_1>0$, the image of $T$ in $\mathrm{Gr}(N,r)$ has degree $2$, i.e. it is a double structure on a line, if and only if $a_1=1$ and $a_i=0$ for all $i>0$.
\end{remark}

The case $a=1$ is very particular because it is the only case with $h^0(\Oo _T)=1$, and it is treated separately below.

\begin{proposition}\label{u5}
$\Dd_{1,4}$ is a smooth and irreducible open subset of $\mathbf{Hilb}_{Q_4}(2,1)$ and $\dim (\Dd_{1,4}) =9$.
\end{proposition}

\begin{proof}
For fixed $A\in \Dd_{1,4}$, we have $\langle A\rangle \cong \PP^2$ since $h^0(\Oo_A(1))=3$. First assume $\langle A\rangle \subset Q$.  In $Q_4$, we have two families of planes, each of them is isomorphic to $\PP^2$, and that the set of all double lines in a plane is isomorphic to a plane; the dual plane of all lines of $\PP^2$. Hence we get in this way two irreducible families $\Tt _1, \Tt _2$ of elements of $\Dd _{1,4}$, each of them of dimension $4$. Fix $A\in \Tt _1$ and set $W:= \langle A\rangle$. Since $A$ is a conic of $W$, we have $N_{A|Q_4}\cong \Oo _A(2)\oplus {(N_{W|Q_4})}_{|A}$. The bundle $N_{W|Q_4}$ is the restriction to $W$ of the universal quotient bundle $\Ee$ on $Q_4$ and so it is spanned. We get $h^1(A,{(N_{W|Q_4})}_{|A}) =0$ and so $h^1(N_{A|Q_4})=0$. We get that $\mathbf{Hilb}_{Q_4}(2,1)$ is smooth and of dimension $h^0(N_{A|Q_4}) = 9$ at $[A]$.

Now fix a line $L\subset Q_4$ and let $H\subset \PP^5$ be a general hyperplane containing $L$. The quadric $Q':= H\cap  Q_4$ is smooth. Fix $o_1, o_2\in L$ with $o_1\ne o_2$ and let $H_i$, $i=1,2$, be the tangent hyperplane $T_{o_i}Q_4$ of $Q_4$ at $o_i$. The scheme $D:= Q_4\cap H\cap H_1\cap H_2$ has $L$ as its reduction and hence it is a complete intersection curve $D\subset \PP^5$ with $D_{\mathrm{red}} =L$ and $\langle D\rangle = H\cap H_1\cap H_2$, i.e. $\dim \langle D\rangle =2$. Hence we get $D\in \Dd_{1,4}$. We also get that $D$ is the complete intersection of $Q_4$ and three hyperplanes, and so $N_{D|Q_4} \cong \Oo _D(1)^{\oplus 3}$. We get that $\mathbf{Hilb}_{Q_4}(2,1)$ is smooth and of dimension $9$ at $[D]$. This part of $\Dd _{1,4}$ is irreducible, because the set of all lines of $Q_4$ is irreducible. Since $\dim (\Tt _i) <9$, we get that $\Dd _{1,4}$ has a unique irreducible component.
\end{proof}

\begin{proposition}\label{xxxx1}
For $\Dd_{1,n}$ with $n\ge 5$, we get the following:
\begin{itemize}
\item[(i)] $\Dd_{1,n}$ is irreducible. 
\item[(ii)] Each $C\in \Dd_{1,n}$ is a flat limit of a smooth conic of $Q_n$.
\item[(iii)] $\mathbf{Hilb}(Q_n)$ is smooth at $C\in \Dd_{1,n}$. 
\end{itemize}
\end{proposition}

\begin{proof}
Fix $C\in \Dd _{1,n}$ and let $M=\langle C \rangle \subset \PP^{n+1}$ denote the plane spanned by $C$. If $M$ is contained in $Q_n$, then we may deform $C$ to a smooth conic inside $M$ and so inside $Q_n$. Since the normal bundle $N_{M|Q_n}$ is globally generated, we have $h^1(C,{\left( N_{M|Q_n}\right)}_{|C})=0$. From $N_{C|M}\cong \Oo _C(2)$ we have $h^1(N_{C|M})=0$. The inclusion $C\hookrightarrow M$ is a regular embedding and so the natural map $N_{C|Q_n} \to {\left( N_{M|Q_n}\right)}_{|C}$ is surjective. Hence the normal sheaf sequence of the inclusions $C\subset M\subset Q_n$ gives $h^1(N_C)=0$ and so $\mathbf{Hilb}(Q_n)$ is smooth at $C$.

If $M$ is not contained in $Q_n$, then we have $C = Q_n\cap M$ as schemes and $N_C\cong \Oo _C(1)^{\oplus (n-1)}$. Thus we have $h^1(N_C)=0$ and so $\mathbf{Hilb}(Q_n)$ is smooth at $C$.
Moving $M$ to a plane in $\PP^{n+1}$ transversal to $Q_n$, we get that $C$ may be deformed inside $Q_n$ to a smooth conic. 
\end{proof}

\begin{proposition}\label{u3}
Let $T=T[a]$ with $a\ge 2$. For $\Ee := \Oo _T(1)\oplus \Oo _T$ and a general $4$-dimensional linear subspace $V$ of $H^0(\Ee)$. Then we have the following: 
\begin{enumerate}
\item $(\Ee ,V)$ induces an embedding of $T$.
\item $\Dd _{a,4}$ is non-empty and irreducible.
\item Every element of $\Dd _{2,4}$ is a flat limit inside $Q_4$ of a family of disjoint unions of two lines and a general element of $\Dd _{2,4}$ is a smooth point of $\mathbf{Hilb}(Q_4)$.
\end{enumerate}
\end{proposition}

\begin{proof}
It is sufficient to prove part (1). Note that $V$ spans $\Ee$ by Remark \ref{u2.1}. 

For each point $p\in \PP^1 = T_{\mathrm{red}}$, let $T_pT$ be the Zariski tangent space of $T$ and $p$ and call $D_p\subset T_pT$ the Zariski tangent space of $T_{\mathrm{red}}$. Then we have $T_pT\setminus D_p\cong \AA ^2$.  For $v\in T_pT\setminus D_p$, the sequence (\ref{eqb1}) gives $h^0(T,\Ii _{v}) = h^0(\Oo _T)-2$ since $a\ge 2$. So we have $h^0(T,\Ii _{v}\otimes \Ee )=h^0(\Ee )-4$. 

Consider $U_v:= H^0(T,\Ii _{v}\otimes \Ee )$ as a $4$-codimensional linear subspace of $H^0(\Ee)$. Let $B_v$ be the Schubert cycle of all $V\in \mathrm{Gr}(4,H^0(\Ee ))$ with $\dim (V\cap U_v) >0$. The variety $B_v$ has codimension $4$ in $\mathrm{Gr}(4,H^0(\Ee ))$. Since $\dim (T)=1$, the set 
$$\bigcup _{p\in \PP^1, ~v\in T_pT\setminus D_p} B_v$$
is not dense in $\mathrm{Gr}(4,H^0(\Ee ))$, and so for a general $V\in \mathrm{Gr}(4,H^0(\Ee ))$ we have $V\cap B_v =\emptyset$ for all $p$ and $v\in T_pT\setminus D_p$. For a general $V$ we may also assume that it spans $\Ee$. Let $f: T\to \mathrm{Gr}(V, 2)$ denote the morphism associated to $(\Ee ,V)$. Note that $f(T_{\mathrm{red}})$ is a line and that $f_{|T_{\mathrm{red}}}$ is injective. Thus $f$ is an embedding if and only if it is a local embedding at all $p\in \PP^1$. By \cite[proof of II.7.3]{Hartshorne} it is sufficient to prove that its differential induces an embedding of the Zariski tangent space $T_pT$ of $T$ at a fixed point $p$. Fix $v\in T_pT\setminus \{0\}$. If $v$ is tangent to $T_{\mathrm{red}}$, then $df _{p\ast}(v)\ne 0$, because $f_{|T_{\mathrm{red}}}$ is an embedding. If $v\notin D_p$, we have $df _{p\ast}(v)\ne 0$, because $V\cap B_v =\emptyset$, i.e. $V\cap H^0(\Ii _v\otimes \Ee)=0$.

Now assume $a=2$. Since $\Dd _{2,4}$ is irreducible, so it is sufficient to prove that a general $C\in \Dd _{2,4}$ satisfies $h^1(N_C)=0$ and that $C$ is a flat limit of a family of disjoint unions of two lines. Fix a smooth hyperplane section $Q_3\subset Q_4$ and any $C$ contained in $\Dd _{2,4}$. By Lemma \ref{a2} $C$ is a flat limit of a family of disjoint unions of two lines of $Q_3$. By the proof of the Claim in the proof of Lemma \ref{a2}, $C$ is contained in a smooth quadric surface $Q_2\subset Q_3$. We have $N_{C|Q_2} \cong \Oo _C$ and so $h^1(N_{C|Q_2}) =0$. Since $h^1(\Oo _C(1)) =0$, we get $h^1(N_C)=0$. By the semicontinuity, we get $h^1(N_D)=0$ for a general $D\in D_{2,4}$.
\end{proof}

\begin{remark}
If $a =2$ (resp. $a=3$), then any $C\in \Dd _{a,4}$ is contained in a $3$-dimensional (resp. $4$-dimensional) linear subspace, because $h^0(\Oo _C(1)) =a+2$.
\end{remark}

\begin{proposition}\label{xxxx2}
For $\Dd_{2,n}$ with $n\ge 5$, we get the following:
\begin{itemize}
\item[(i)] $\Dd _{2,n}$ is irreducible. 
\item[(ii)] Each $C\in \Dd _{2,n}$ is a flat limit of a family of two disjoint lines.
\item[(iii)] $\mathbf{Hilb}(Q_n)$ is smooth at $C\in \Dd_{2,n}$, outside the $(4n-9)$-dimensional family of curves in $\Dd _{2,n}$ whose linear span intersects $Q_n$
in the double plane.
\end{itemize}
\end{proposition}

\begin{proof}
Let $M=\langle C \rangle \subset \PP^{n+1}$ denote the linear span of $C\in \Dd_{2,n}$ and then we have $\dim (M)=3$.

First assume $M\subset Q_n$. In this case $C$ may be deformed inside $M$ to a disjoint union of two lines. Since the normal bundle $N_{M|Q_n}$ is globally generated, so we have $h^1(C,\left( N_{M|Q_n}\right)_{|C})=0$. We also have $h^1(N_{C|M}) =0$, because $C$ is contained in a smooth quadric surface contained in $M$. Since $C$ is a locally complete intersection in $M$, the natural map  $N_{C|Q_n} \to \left(N_{M|Q_n}\right)_{|C}$ is surjective. Hence the normal sheaves sequence of the inclusions $C\subset M\subset Q_n$ gives $h^1(N_C)=0$.

Now assume that $M\cap Q_n$ is a quadric surface with singular locus of dimension at most $1$, i.e. $M$ is not a double plane. In this case there is a smooth $4$-dimensional quadric $Q_4$ such that $M\cap Q_n \subset Q_4\subset Q_n$. By part (3) of Proposition \ref{u3}, $C$ is smoothable to a disjoint union of two lines and there is a smooth quadric surface $Q_2\subset Q_4$ such that $C\subset Q_2$. We first get $h^1(N_{C|Q_2}) =0$ and then $h^1(N_C)=0$.

Assume that $M\cap Q_n$ is a double plane. The family of all curves in $\Dd_{2,n}$, which is contained in $M\cap Q_n$, is irreducible and $5$-dimensional by \cite[Corollary 4.3]{hs} with $Z =\{q\}$, $Y= P=L$ and $z=y=p=1$. Since the set of all planes contained in $Q_n$ is an irreducible variety of dimension $2 +4(n-4) =4n-14$, we get that the family of all such curves on double planes is parametrized by an irreducible variety of dimension $5+(4n-14) < 2+4(n-2) = \chi (N_C)$. Since $C$ is a locally complete intersection with embedded dimension $2$ everywhere, so it deforms to another curve $C'$ of embedded dimension at most $2$ everywhere, which is locally a complete intersection and so with no embedded point. In particular $C'$ is either a disjoint union of two lines or a double structure on a line with no embedded point and with embedded dimension $2$ everywhere. In the latter case it is a ribbon because of its embedded dimension $2$ everywhere. Hence we get that $C'\in \Dd _{2,n}$ and we many assume that $C'$ spans a $3$-dimensional linear space whose intersection with $Q_n$ is not a double plane because of dimensional reason. Since $C'$ is a limit of a family of two disjoint lines, so is $C$.  
\end{proof}

%%%%%%%%%%%%%%%%%%%%%%%%%%%%%

\section{Warm-up case of $\chi=1$}

\begin{remark}\label{a1.0}
Let $A\subset \PP^r$, $r\ge 3$, be a locally Cohen-Macaulay scheme of degree $2$ with pure dimension $1$. Then we have $\chi (\Oo _A)\ge 1$ by the upper bound for the genus of locally Cohen-Macaulay curves with degree $2$ in projective space of dimension at least three (e.g. see \cite{h1}). Since $\deg (A)=2$, this may also be proved using that no ribbon of positive genus has a very ample line bundle of degree $2$. Thus we have $\mathbf{Hilb}(2,a)=\emptyset$ for all $a\le 0$.
\end{remark}

\begin{lemma}
We have $\Dd_1 \cong \PP^3$.
\end{lemma}
\begin{proof}
It is enough to check that $\Dd_1(L)$ is a single point space for each line $L\subset Q$. The elements of $\Dd _1$ are the complete intersection of two quadric cones of $Q$ such that their vertices are contained in a line $L\subset Q$. Indeed, fix any line $L\subset Q$ and two distinct points $p,q\in L$. The conic $T_pQ\cap T_qQ\cap Q$ contains $L$ and it has at least two singular points $p$ and $q$. Thus we have $T_pQ\cap T_qQ\cap Q\in \Dd _1$. Using the quadratic form associated to $Q$, we see that the map $u: Q\to (\PP^{4})^{\vee }$ defined by $p\mapsto T_pQ$ is injective with a smooth quadric surface $u(Q)$ as its image. Thus for all $p, q\in L$ with $p\ne q$, we have $T_pQ\ne T_qQ$ and so $T_pQ\cap T_qQ\cap Q\in \Dd _1(L)$. 

\quad {\emph {Claim}}: The set $u(L)$ is a line, i.e. there is a unique plane $E\subset \PP^4$ such that $T_pQ\supset E$ for all $p\in L$.

\quad {\emph {Proof of the Claim}}: Since the set of all lines of $Q$ is homogeneous for the action of the group $\mathrm{Aut}(Q)$, it is sufficient to prove the Claim for a single line. Fix homogeneous coordinates $[x_0:\cdots:x_4]$ such that $Q =\{x_0x_1+x_2x_3+x_4^2=0\}$ and $L =\{x_0 =x_2=x_4=0\}$. For $p =[0:1:0:0:0]$, we have $T_pQ =\{x_0=0\}$. If $q = [0:a:0:1:0]$ for some $a$, then we have $T_qQ = \{ax_0+x_2=0\}$. Thus we may choose $E = \{x_0=x_2=0\}$.  \qed

Since $T_pQ\cap T_qQ =E$ for all $p\ne q\in L$ in the Claim, we get that $\Dd _1(L)$ is a unique point.
\end{proof}

\begin{lemma}\label{a1.1}
We have $\mathbf{Hilb}(2,1) = \mathbf{Hilb}(2,1)_+\cong \mathrm{Gr}(3,5)$, the Grassmannian variety parametrizing projective planes in $\PP^4$.
\end{lemma}

\begin{proof}
Recall that $Q$ contains no plane and so the curve $C:= M\cap Q$ for a plane $M$ is an element
of $\mathbf{Hilb}(2,1)_+$. Conversely, for a fixed $[C]\in \mathbf{Hilb}(2,1)_+$, either $C$ is a reduced conic or $C\in \Dd _1$. In the latter case $C$ fits into the exact sequence
$$0\to \Oo_L(-1) \to \Oo_C \to \Oo_L \to 0$$
with $L:=C_{\mathrm{red}}$ and so $h^0(\Oo _C(1))=3$. Thus in either case $C$ is contained in a plane $\langle C \rangle \cong \PP^2$ and so we have $C=\langle C \rangle \cap Q$ as schemes. Thus we have $(\mathbf{Hilb}(2,1)_+)_{\mathrm{red}} \cong \mathrm{Gr}(3,5)$. Now we check that
$\mathbf{Hilb}(2,1) = \mathbf{Hilb}(2,1)_+$. Fix $[C]\in \mathbf{Hilb}(2,1)$ and let $D\subseteq C$ be the maximal locally Cohen-Macaulay subcurve of $C$; the ideal sheaf $\Ii _{D,C}$ of $D$ in $C$ is the intersection of the non-embedded components of a primary decomposition of the Noetherian sheaf
$\Oo _C$. Let $a$ be the degree of the kernel of the quotient map $\Oo _C\to \Oo _D$. If $D\ne C$, i.e. $[C]\notin  \mathbf{Hilb}(2,1)_+$, then we have $a>0$ and so $D\in  \mathbf{Hilb}(2,1-a)_+$. Thus we get $\mathbf{Hilb}(2,1)_{\mathrm{red}}=\mathbf{Hilb}(2,1)_{+, \mathrm{red}}\cong \mathrm{Gr}(3,5)$ by Remark \ref{a1.0}. Now for each conic $C\subset Q$, we get the locally free resolution of $\Ii_C$,
$$0\to \Oo_Q(-2) \to \Oo_Q(-1)^{\oplus 2} \to \Ii_C \to 0.$$
Tensoring it with $\Oo_C$, we get $N_{C|Q} \cong \Oo_C(1)^{\oplus 2}$. In particular we have $h^0(N_{C|Q})=6$ and $h^1(N_{C|Q})=0$. Hence we have $\mathbf{Hilb}(2,1) \cong \mathbf{Hilb}(2,1)_{\mathrm{red}}$ and the assertion follows. 
\end{proof}

\begin{remark}
Let $\GG_i$ for $i=1,2,3$ be the subvariety of $\mathrm{Gr}(3,5)$ parametrizing conics on $Q$ of rank at most $i$. In particular, we have $\GG_1 \cong \PP^3$ and $\GG_3=\mathrm{Gr}(3,5)$. To each point $p\in Q$ we may associate the projective plane formed by the planes contained in $T_pQ$ and containing $p$. Thus $\GG_2$ is a hypersurface and so it is an element of $|\Oo_{\mathrm{Gr}(3,5)}(a)|$ for some $a>0$. Fix a general hyperplane $H\subset \PP^4$ and set $Q_1:= H\cap Q$ a smooth quadric surface. Let $M \subset \mathrm{Gr}(3,5)$ be the set of all planes contained in $H$ and then $\GG_2\cap M$ is the set of all planes in $H$ which are tangent to $Q_1$, i.e. the dual variety $Q_1^\vee$ of $Q_1$, as a hypersurface of $H$. Since $Q_1^\vee$ is a smooth quadric surface, so we get $a=2$. $\GG_2$ is homogeneous for the action of $\mathrm{Aut} (Q)$, because this group acts transitively on the set of all reducible conics; it acts transitively on the set of singular points of conics and if $D_1, D_2$ are reducible conics with common vertex point $p$, then they correspond to pairs of different lines of the surface quadric cone $T_pQ\cap Q$.
\end{remark}

\begin{proposition}\label{a1}
We have 
$$\mathbf{M}(2,1) \cong \mathbf{M}^{\alpha} (2,1) \cong \mathbf{Hilb}(2,1) \cong \mathrm{Gr}(3,5)$$ 
for all $\alpha>0$ and they are fine moduli spaces. 
\end{proposition}

\begin{proof}
The fineness of $\mathbf{M}(2,1)$ comes from \cite[Corollary 4.6.6]{HL} with $(a_0,a_1)=(1,2)$. Let $(1,\Ff)\in \mathbf{M}^{\alpha}(2,1)$ be a strictly $\alpha$-semistable pair and so it has a subpair $(s,\Ff')$ with $\chi_{\Ff'}(t)=t+c$ such that $c+\delta \cdot \alpha=\frac{1+\alpha}{2}$. In particular we have $\Ff' \cong \Oo_{L'}(c-1)$ for a line $L'\subset Q$. If $\delta=1$, then we have $2c+\alpha=1$. But since $\Ff'$ has a non-zero section, so we have $c\ge 1$, a contradiction. If $\delta=0$, then we have $c=\frac{1+\alpha}{2}$. But the quotient pair $(1,\Ff'')$ has $\chi_{\Ff''}(t)=t+1-c$. Since $\Ff'' \cong \Oo_{L''}$(-c) with a line $L''$, so we have $c\le 0$, a contradiction. Thus there is no wall-crossing among $\mathbf{M}^{\alpha}(2,1)$'s and so we have
$$\mathbf{M}^{\infty}(2,1) \cong \mathbf{M}^{\alpha}(2,1) \cong \mathbf{M}^{0+}(2,1)$$
for all $\alpha>0$. Note that a curve $C$ in $\mathbf{Hilb}(2,1)$ is a conic, including a planar double line. For such a curve $C$, we have that $\Oo_C$ is stable with $\chi(\Oo_C)=1$ and $h^0(\Oo_C)=1$. Note that $C$ is either
\begin{itemize}
\item a smooth conic with ${TQ}_{|C} \cong \Oo_{\PP^1}(2)^{\oplus 3}$ and so $N_{C|Q} \cong \Oo_{\PP^1}(2)^{\oplus 2}$, or
\item a reducible conic, say $C=L_1 \cup L_2$ (possibly $L_1=L_2$), with $N_C \cong \Oo_C(1)^{\oplus 2}$
\end{itemize}
and so it is evident that $\Oo_C$ is stable. Conversely, for $\Ff \in \mathbf{M}(2,1)$, the curve $C=C_{\Ff}$ has degree at most $2$. If $\deg (C)=1$, say $C=L$ a line, then $\Ff$ is a vector bundle of rank $2$ on $L$, i.e. $\Ff \cong \Oo_L(a_1)\oplus \Oo_L(a_2)$ with $a_1+a_2=-1$ and $a_1>a_2$. The subbundle $\Oo_L(a_1)$ destabilize $\Ff$ and so we have $\deg (C)=2$. Since $\chi(\Ff)=1$, so we have $h^0(\Ff)>0$ and a non-zero section induces an injection $u:\Oo_D \hookrightarrow \Ff$ for a subcurve $D\subset C$. In the case of $D=C$, we have $[D]\in \mathbf{Hilb}(2,1)$ by Remark \ref{a1.0} and $\mathrm{coker}(u)=0$. If $D\subsetneq C$, then $D=L$ is a line. Since $\chi_{\Oo_L}(t)=t+1$, so it contradicts to the semistability of $\Ff$. Thus the forgetful morphism $\phi : \mathbf{M}^{\infty}(2,1) \to \mathbf{M}(2,1)$ is bijective. To confirm that it is an isomorphism, it is enough to check the smoothness of $\mathbf{M}(2,1)$. Note that $\dim \Hom_Q(\Oo_C, \Oo_C)=1$, since $C$ is connected. Apply $\Ext^{\bullet}_Q(-, \Oo_C)$ to the standard sequence for $C\subset Q$, we get
\begin{align*}
0&\to \Hom_Q(\Oo_C, \Oo_C) \to \Hom_Q(\Oo_Q, \Oo_C) \to \Hom_Q(\Ii_C, \Oo_C)\\
&\to \Ext^1_Q(\Oo_C, \Oo_C) \to \Ext^1_Q(\Oo_Q, \Oo_C).
\end{align*}
Since $\Ext^1_Q(\Oo_Q, \Oo_C)\cong H^1(\Oo_C)=0$, so we get $\Ext^1_Q(\Oo_C, \Oo_C) \cong \Hom_Q(\Ii_C, \Oo_C)$. But we have $\Hom_Q(\Ii_C, \Oo_C) \cong \Hom_C (\Ii_C \otimes \Oo_C, \Oo_C) \cong H^0(N_{C|Q})$ and so its dimension is $6$. In particular, $\mathbf{M}(2,1)$ is smooth.

The morphism from $\mathbf{M}^{\infty}(2,1) \to \mathbf{Hilb}(2,1)$ is bijective and so it is an isomorphism since $\mathbf{Hilb}(2,1)$ is smooth. In particular, $\mathbf{M}^{\infty}(2,1)$ is smooth. 

Now let $\mathcal{U}$ be the universal sheaf on $Q\times \mathbf{M}(2,1)$ and then $f_{1*}(\mathcal{U})$ is a line bundle on $Q$, where $f_1$ is the projection to $Q$. If $g_1: Q \times \mathbf{M}^{\infty}(2,1) \to Q$ is the projection, then the pair $(g_1^*f_{1*} (\mathcal{U}), \phi^*\mathcal{U})$ is the universal coherent system. 
\end{proof}

%%%%%%%%%%%%%%%%%%%%%%%%

\section{Case of $\chi=2$}

\begin{lemma}\label{a2}
$\Dd_2$ is an irreducible and non-empty subset of $\mathbf{Hilb}(2,2)$ with $\dim (\Dd _2) = 5$. For any $C\in \Dd_2$, we have the following:
\begin{itemize}
\item[(i)] $C$ is contained in a smooth hyperplane section $Q_2$.
\item[(ii)] $h^1(N_C)=0$ and $h^1(\Ii _C(t)) =0$ for all $t>0$. In particular, $C$ is a smooth point of $\mathbf{Hilb}(2,2)$.
\item[(iii)] $C$ is a flat limit of a family of disjoint unions of two lines.
\end{itemize}
\end{lemma}

\begin{proof}
Note that for each smooth hyperplane section $Q_2\subset Q$ and any line $L\subset Q_2$, the effective Cartier divisor $2L$ of $Q_2$ is a split ribbon with the prescribed $\chi (\Oo _C)$ and so it is an element of $\Dd _2$. In particular we have $\Dd_2 \ne \emptyset$. Conversely fix any $C\in \Dd _2$ and then it fits into the exact sequence
\begin{equation}\label{equuu1}
0\to \Oo_L \to \Oo_C \to \Oo_L\to 0\end{equation}
with $L:=C_{\mathrm{red}}$. Since $h^0(\Oo _C(1)) =4$, so $C$ is contained in the scheme-theoretic intersection $M\cap Q$, where $M\subset \PP^4$ is a hyperplane. Note that the hyperplane $M$ is unique, because it is the linear span $\langle C \rangle$ of $C$ in $\PP^4$. 

\quad {\emph {Claim }}: $M$ is not a tangent hyperplane of $Q$, i.e. $M\cap Q$ is a smooth quadric surface.

\quad {\emph {Proof of Claim }}: Let $A$ be the split ribbon with $\chi (\Oo _A)=2$. The unique line bundle $\Ll$ of degree $2$ on $A$ is very ample and $h^0(\Ll )=4$. Therefore, up to projective transformations, there is a unique ribbon $E\subset \PP^3$ with $\chi (\Oo _E)=2$. Set $L:= E_{\mathrm{red}}$. Any smooth quadric surface containing $L$ also contains an embedding $E$ of $A$ with $L$ as its support. By the uniqueness of the embedding of $A$ up to a projective transformation, we get $h^1(\PP^3,\Ii _E(2)) =0$, $h^0(\PP^3, \Ii_E(2))=4$ and that a general quadric surface containing $E$ is smooth. $E$ is contained in all quadric surfaces which are the union of two planes through $L$. We need to prove that $E$ is not contained in an integral quadric cone. Assume $E\subset U$ with $U$ a quadric cone and call $o$ the vertex of $U$. Then we have $o\in L$. Since $U$ is a quadric cone, so for two points $p, p'\in L\setminus \{o\}$ the planes $T_pU$ and $T_{p'}U$ are the same. Since $E$ has planar singularities, we get that $T_pU$ is the Zariski tangent plane to $E$ at $o$. Fix a smooth quadric surface $Q_2\supset E$. Since $E$ has planar singularities, we get $T_pQ = T_pE$ for all $p\in L$. The tangent planes $T_pQ_2$ with $p\in L\setminus \{o\}$, have the property that $T_pQ_2\cap Q = L\cup L_p$, where $L_p$ is the unique line of $Q_2$ intersecting $L$ and containing $p$. Hence $p\ne p'$ would imply $T_pQ_2\ne T_{p'}Q_2$, a contradiction. \qed

By Claim, $C$ is contained in a smooth quadric surface and so we get $h^1(N_C)=0$. Thus $[C]$ is a smooth point of $\mathbf{Hilb}(2,2)$. Similarly we get $h^1(\Ii _C(t)) =0$ for all $t>0$. We also see that the algebraic set $\Dd _2$ is equidimensional of dimension $5$ and that it has one or two irreducible components, because each smooth quadric surface has exactly two rulings. The set of all lines in $Q$ is isomorphic to $\PP^3$. In particular it is irreducible and of dimension $3$. For each line $L\subset Q$, the set of all hyperplanes $M\subset \PP^4$ containing $L$ such that $M\cap Q$ is smooth is a non-empty open subset of $\PP^2$. Thus it is irreducible of dimension $2$ and so $\Dd _2$ is irreducible.
\end{proof}

\begin{remark}\label{rem443}
Note that the set of hyperplanes in $\PP^4$ whose intersection with $Q$ is a smooth quadric surface, is parametrized by ${\PP^4}^\vee \setminus Q^\vee$, the dual space of $\PP^4$ minus the dual of $Q$. The quadratic form associated to $Q$ induces an isomorphism between $Q$ and it dual $Q^\vee$, which associates $T_pQ$ to each $p\in Q$; under the isomorphism, a line $L\subset Q$ is mapped to a line in $Q^\vee$, not just a rational curve. In particular, we have $\Dd_2(L) \cong \PP^2 \setminus \PP^1$. 
\end{remark}

\begin{proposition}\label{a3}
$\mathbf{Hilb}(2,2)_+$ is smooth, irreducible and of dimension $6$.
\end{proposition}

\begin{proof}
Take $[C]\in \mathbf{Hilb}(2,2)_+$. Then we have that either $C$ is the disjoint union of two lines or $C\in \Dd _2$. If $C$ is the disjoint union of two lines, then we have $h^1(N_C)=0$ and so the set of
all such curves is irreducible of dimension $6$. Now the assertion follows from Lemma \ref{a2}.
\end{proof}

\begin{remark}\label{rem33}
The scheme $\mathbf{Hilb}(2,2)_+$ is not closed in $\mathbf{Hilb}(2,2)$ and in particular it is not complete. Let $T\subset Q$ be a surface quadric cone with the vertex $o$. Fix two distinct lines $L, R\subset T$  and so we have $L\cap R=\{o\}$. Let $\Delta$ be an integral affine curve with a fixed point $q\in \Delta$ and $\{R_p\}_{p\in \Delta}$ be a family of lines of $Q$ with $R_q =R$ and $L \cap R_p =\emptyset $ for all $p\in \Delta \setminus \{q\}$. The family $\{L\cup R_p\}_{p\in \Delta \setminus \{q\}}$ has a flat limit $B\in \mathbf{Hilb}(2,2)$ containing $L \cup R$, but with an embedded point $o$. Note that $B$ is not contained in $T$. 
\end{remark}

Fix $[C]\in  \mathbf{Hilb}(2,2)\setminus  \mathbf{Hilb}(2,2)_+$. Let $D\in \mathbf{Hilb}(2,c)_+$ be the maximal locally Cohen-Macaulay subscheme of $C$ with pure dimension $1$. Let $a$ be the degree of the kernel of the surjection $\phi : \Oo _C\to \Oo _D$ and then we have $c=2-a$. Remark \ref{a1.0} gives $a=1$ and $c=1$. Thus the sheaf $\mathrm{ker}(\phi)$ is the structural sheaf $\CC _p$ of a unique point $p\in Q$ and we have a map 
\begin{equation}\label{chow}
\psi : \mathbf{Hilb}(2,2) \to \mathbf{Hilb}(2,2)_+ \sqcup \left (\mathbf{Hilb}(2,1) \times Q \right).
\end{equation}
For $(D,p)\in \mathbf{Hilb}(2,1)_+ \times Q$, let us define $\Aa (D,p):=\psi^{-1}(D,p)$, i.e. the set of 
all $[C]\in  \mathbf{Hilb}(2,2)$ with $D$ as the maximal locally Cohen-Macaulay subscheme of $C$ with pure dimension $1$ and with $p\in D_{\mathrm{red}}$ as the support of the kernel of the surjection $\Oo _C \to \Oo _D$. If $p\notin D_{\mathrm{red}}$, then we have $C \cong D\sqcup \{p\}$ as schemes and so $\Aa(D,p)$ is a single-point space. Thus it is sufficient to consider the scheme $C$ with $p\in D_{\mathrm{red}}$.

\begin{lemma}\label{aa22}
For $D\in \mathbf{Hilb}(2,1)$ and $p\in D_{\mathrm{red}}$, we have $\Aa(D,p) \cong \PP^1$. 
\end{lemma}
\begin{proof}
Assume that $D$ is smooth at $p$. In the local ring $\Oo _{Q,p}$, there are generators $\{x,y,z\}$ of the maximal ideal $\mathfrak{m}$ of $\Oo _{Q,p}$ such that $D$ has local equations $y=z=0$. There is a bijection between $\Aa (D,p)$ and the set of all ideals $J$ of $\Oo _{Q,p}$ contained in $(y,z)$ and with a $1$-dimensional vector space $(y,z)/J$. In particular, $J$ contains a linear combination of $y$ and $z$. Thus we may assume that $z\in J$ and so we get $(y,z)/J \cong (y)/J'$ for an ideal $J'\subset \CC[x,y]$. The only possibility of such $J'$ is $(xy,y^2)$ and so the set of ideals $J$ is parametrized by the set of planes containing the line $y=z=0$. In other words, letting $T_pD \subset T_pQ$ be the tangent line of $D$ at $p$, there is a bijection between $\Aa(D,p)$ and the planes in $T_pQ$ containing $T_pD$. Hence we have $\Aa(D,p) \cong \PP^1$. 

Assume now that $D$ is reduced and that $p$ is a singular point of $D$. In this case $D$ is a reduced conic with the singular point $p$. We may assume that $D$ is locally defined by $z=xy=0$ with the maximal ideal $\mathfrak{m}=(x,y,z)$ of $\Oo_{Q,p}$. Again there is a bijection between $\Aa(D,p)$ and the set of ideals $J$ contained in $(xy,z)$ and with a $1$-dimensional vector space $(xy,z)/J$. Thus $J$ contains a linear combination of $xy$ and $z$, say $axy+bz\in J$. If $a=0$, i.e. $z\in J$, then $C$ is contained in the plane $z=0$, then we have $(xy,z)/J \cong (xy)/J'$ for an ideal $J' \subset \CC[x,y]$. Since $\dim_{\CC} (xy)/J'=1$, so the only possibility is $J'=(x^2y,xy^2)$. If $a\ne 0$, then we have $xy+cz \in J$. Then we must have an isomorphism $(xy,z)/J \cong \CC\{z\}$ defined by $xy \mapsto -cz$, since we have $\dim_{\CC}(xy,z)/J=1$. Thus we have $\{z^2,xz,yz\} \subset J$ and so we get $\{xy+cz, z^2, xz, yz\} \in J$. Indeed the ideal $(xy+cz, z^2, xz, yz)$ defines $D$ with the embedded point $p$ and so we have $J=(xy+cz, z^2, xz, yz)$. Overall we get $J=(axy+bz, z^2, xz, yz,x^2y, xy^2)$ and thus the ideals $J$ are parametrized by the choice of coordinates $(a: b) \in \PP^1$ and so we get $\Aa(D,p) \cong \PP^1$. In fact, there are two kinds of curves in $\Aa(D,p)$, one with $a=0$ and the others with $a\ne 0$. The curves with $a\ne 0$ are isomorphic to each others, but they define different elements in $\mathbf{Hilb}(2,2)$.

Now assume $D\in \Dd_1$. Since there exists a unique plane $\langle D \rangle \cong \PP^2$ containing $D$, we may assume that $D$ has local equation $z=x^2=0$, i.e. $\langle D \rangle$ is defined by $z=0$. The ideal $J$ defining $C$ is contained in $(z,x^2)$ such that the dimension of the vector space $(z,x^2)/J$ is $1$. Similarly as in the previous case, if $z\in J$, then we have $(z,x^2)/J \cong (x^2)/J'$ for an ideal $J'\subset \CC[x,y]$. Thus we have $J'=(x^3, x^2y)$. If $z\not \in J$, then we have an isomorphism $(z,x^2)/J \cong \CC\{z\}$ defined by $x^2 \mapsto -cz$ for some $c\in \CC$ and so we get $J=(x^2+cz, z^2, xz, yz)$. Overall we get $J=(ax^2+bz, z^2, xz,yz,x^3, x^2y)$ and thus the ideals $J$ are parametrized by the choice of coordinates $(a:b) \in \PP^1$ for which $ax^2+bz\in J$ and so we get $\Aa(D,p)\cong \PP^1$. 
\end{proof}

Let $\II_1$ be the incidence variety of $\mathbf{Hilb}(2,1)\times Q$, i.e. 
$$\II_1:=\{(D,p)\in \mathbf{Hilb}(2,1)\times Q~|~ p\in D\}.$$
The set of conics on $Q$ passing through a fixed point $p\in Q$ is isomorphic to $\mathrm{Gr}(2,4)$ and so $\pi_2 : \II_1 \to Q$ is a $\mathrm{Gr}(2,4)$-fibration. In particular $\II_1$ is a smooth variety of dimension $7$. Let us define $\II_2$ to be the diagonal over $\mathrm{Sym}^2(\PP^3)$, the second symmetric power of $\PP^3$.

Let $\Hh_1$ be the subvariety of $\mathbf{Hilb}(2,2)$ parametrizing the non-locally Cohen Macaulay curves and then we have a morphism $\psi_{|\Hh_1} : \Hh_1 \to \mathbf{Hilb}(2,1)\times Q$. We also let $\Hh_2$ be the closure of $\mathbf{Hilb}(2,2)_+$. 

Let us set $\Hh' :=(\Hh_1 \cap \Hh_2)_{\mathrm{red}}$. If $D\in \mathbf{Hilb}(2,1)$ and $p\notin D_{\mathrm{red}}$, then $C:= D\cup \{p\}$ is a smooth point of $\mathbf{Hilb}(2,2)$.Therefore we have $C\notin \Hh'$.

\begin{lemma}\label{deg}
Let $D$ be a conic in $\mathbf{Hilb}(2,1)$ and $p$ a point on $D$. 
\begin{enumerate}
\item If $p\in D_{\mathrm{sm}}$, we have $\Aa(D,p) \cap \Hh' = \emptyset$, i.e. no curve in $\Aa(D,p)$ is a flat limit of $\mathbf{Hilb}(2,2)_+$. 
\item If $p\in D_{\mathrm{sing}}$, we have $\Aa(D,p) \cap \Hh'\ne \emptyset$, i.e. there exists a curve $C\in \Aa(D,p)$ which is a flat limit of $\mathbf{Hilb}(2,2)_+$. 
\end{enumerate}   
\end{lemma}
\begin{proof}
Since $\mathrm{Gr}(3,5)$ is complete, at least one irreducible component of $D_{\mathrm{red}}$ is a line and so $D$ is a reducible conic. Let $e$ be the singular point of $D$. Assume the existence of a flat family $\pi :\Gamma \to T$ with $T$ an irreducible curve, $o\in T$, $\pi ^{-1}(o) =C$ and $\pi ^{-1}(t)\in \mathbf{Hilb}(2,2)_+$ for all $t\in T\setminus \{o\}$. Since the support of the nilradical of $\Oo _C$ is a single point $p$,t hen  $\pi ^{-1}(t)\notin \Dd _2$ for a general $t\in T$. Since $p\ne e$, there is an open neighborhood $U$ of $e$ in $\Gamma$ such that $\pi_{|U}: U \to \pi (U)$ has reduced fibers, one connected fiber with an ordinary node and a general fiber a smooth curve with two connected components because which is impossible by \cite[Proposition X.2.1]{ACG}. 

\quad{(a)} Assume that $D$ is a reducible conic with the singular point $p \ne e\in D$. Since the support of the nilradical of $\Oo _C$ is a single point $p$,t hen  $\pi ^{-1}(t)\notin \Dd _2$ for a general $t\in T$. Since $p\ne e$, there is an open neighborhood $U$ of $e$ in $\Gamma$ such that $\pi_{|U}: U \to \pi (U)$ has reduced fibers, one connected fiber with an ordinary node and a general fiber a smooth curve with two connected components because which is impossible by \cite[Proposition X.2.1]{ACG}. 

\quad{(b)} Assume that $D$ is a reducible conic and let $p$ be its singular point. Write $D = L\cup R$ with $L, R$ lines. Let $\{R_t\}_{t\in T}$ be a family of lines of $Q$
with $T$ an integral curve, $R_o = R$ for some $o\in T$ and $R_t\cap D =\emptyset$ for all $t\in T\setminus \{o\}$. Since $\mathbf{Hilb}(2,2)$ is proper, the family
$\{L\cup R_t\}_{t\in T\setminus \{o\}}$ has a limit $C\in \Hh '$. We have $\psi ({C}) = (D,p)$.

\quad{({c})} For a fixed line $L\subset Q$ and any $p, p'\in D$, set $C' := Q\cap T_pQ\cap T_{p'}Q\in \Dd_1$ that is the only element of $\Dd _1$ with $L$ as its reduction. Take a flat family of lines $\{R_\lambda\}$ of $Q$ disjoint from $L$ and with $L$ as its flat limit. Taking $\{R_\lambda \cup L\}$, we get $C\in \Hh '$ with $\psi ({C}) =(C',p'')$ for some $p''\in L$. Since $\mathrm{Aut}(Q)$ acts transitively on the set of all $(L,p'')$ we get that for each $p_1\in L$ there is $C_1\in \Hh '$ with $\psi (C_1) =(C',p_1)$.
\end{proof}

\begin{corollary}\label{rer}
For a singular conic $D$ in $Q$ and a point $p\in D_{\mathrm{sing}}$, we have
$$\left|\Aa(D,p) \cap \Hh'\right|=1.$$
\end{corollary}
\begin{proof}
By (2) of Lemma \ref{deg}, we have $|\Aa(D,p) \cap \Hh_2 | \ge 1$. From the exact sequence
\begin{equation}\label{jkl}
0\to \CC_p \to \Oo_C \to \Oo_D \to 0
\end{equation}
for $C\in \Aa(D,p)$, we have $h^0(\Oo_C(1))=4$ and so we get $\langle C \rangle \cong \PP^3$. In particular, $C$ is contained in a quadric surface $Q_2:=Q \cap \langle C \rangle$. By the proof of \cite[Lemma 2]{PS}, any surface containing $C$ is singular at $p$. In particular, $Q_2$ is a singular quadric surface with the singular point $p$. Without loss of generality, we may assume that $Q_2$ is defined by the equation $z^2-xy=0$ in local coordinates $x,y,z$ of $\langle C \rangle$ and $p=(0,0,0)$. 

If $D$ is a singular conic with the singular point $p$, then we may assume that $D$ is defined by the ideal $(z,xy)$. We know from Lemma \ref{aa22} that the ideal defining $C\in \Aa(D,p)$ is $(axy+bz,z^2,xz,yz,x^2y,xy^2)$ and the only one containing $z^2-xy$ is with $b=0$, i.e. $(z^2,xy,yz,zx)$. 

If $D$ is in $\Dd_2(L)$ with $L$ defined by $(z,x)$, then $D$ is defined by the ideal $(x,z^2)$. Again the ideal defining $C\in \Aa(D,p)$ is $(az^2+bx,x^2,xz,xy,z^3,yz^2)$ and the only one containing $z^2-xy$ is with $b=0$, i.e. $(z^2,x^2,xz,xy)$. 

Hence the curves in $(\Hh_1 \cap \Hh_2)_{\mathrm{red}}$ are determined by the pairs $(D,p)$ with $D\in \GG_2$ and $p\in D_{\mathrm{sing}}$. 
\end{proof}

\begin{remark}
By Corollary \ref{rer}, the curves in $\Aa(D,p)\setminus \Hh'$ with $p\in D_{\mathrm{sing}}$ is parametrized by $\PP^1$ minus a point. Indeed, such curves are given as follows:
\begin{enumerate}
\item Fix a reducible conic $D\subset Q$ with the singular point $p$ and let $Q_2 \subset Q$ be any smooth quadric surface containing $D$ and then we have $D\in |\Oo _{Q_2}(1,1)|$. Let $E\in |\Oo _{Q_2}(1,1)|$ be any smooth conic containing $p$ and let $3p\subset Q_2$ denote the closed subscheme of $Q_2$ with $(\Ii _{p,Q_2})^3$ as its ideal sheaf. The scheme $(D\cup E)\cap (D\cup 3p)$ has $D$ as its reduction, but it contains a degree $3$ subscheme of $E$ with $p$ as its support, while $E\cap D$ is exactly the degree $2$ subscheme of $E$ with $p$ as its reduction.
\item For a double line $D\in \Dd _1$, fix any quadric cone $Q_2 \subset Q$ containing $D$ and take any point $p$ different from the vertex $o$ of $Q_2$. Let $M\subset H$ be a plane containing $p$, but not the line $D_{\mathrm{red}}$. Set $E:= Q_2\cap M$. You take as $C$ the union of $D$ and the degree $3$ subscheme of the smooth conic $E$ with $p$ as its support.
\end{enumerate}
\end{remark}

\begin{proposition}
We have 
$$\mathbf{Hilb}(2,2)=\Hh_1 \cup \Hh_2,$$
with two irreducible components $\Hh_i$ for $i=1,2$ such that
\begin{itemize}
\item[(a)] $\Hh_1$ and $\Hh_2$ are rational varieties of dimension $9$ and $6$ respectively, 
\item[(b)] $(\Hh_1 \cap \Hh_2 )_{\mathrm{red}}$ consists of singular conics $D$ with an extra point $p\in D_{\mathrm{sing}}$ such that the hyperplane section containing the curve is singular at $p$, and 
\item[(c)] $\mathbf{Hilb}(2,2)$ is smooth outside $\Hh_1 \cap \Hh_2$.  
\end{itemize}
\end{proposition}

\begin{proof}
By Proposition \ref{a3} and Corollary \ref{rer}, it remains to show the irreducibility of $\Hh_1$ and its smoothness over  $(\mathbf{Hilb}(2,1)\times Q)\setminus (\Hh_1 \cap \Hh_2)$. 

Let $\psi=\psi_{|\Hh_1} : \Hh _1 \to \mathbf{Hilb}(2,1) \times Q$ be the map defined from (\ref{chow}). We saw in Lemma \ref{aa22} that $\psi$ is surjective and that if $(D,p)\notin \II_1$, then the fiber is a single point, while for each $x = (D,p)\in \II _1$ we have $\psi ^{-1}(x) \cong \PP^1$ and in particular the fiber it is irreducible and of dimension $1$. Hence $\psi ^{-1}(\II _1)$ is irreducible and of dimension $8$. Thus to prove the assertion it is sufficient to prove that a general element of $\psi ^{-1}(\II _1)$ is contained in the closure of $\psi ^{-1} ((\mathbf{Hilb}(2,1) \times Q )\setminus \II_1)$.

Fix $(D,p)\in  \mathbf{Hilb}(2,1) \times Q$ with $D$ a smooth conic and $p\in D$. Let $T_pD$ be the tangent line to $D$ at $p$. For a general $(D,p)$ we may assume that the line $T_pD$ is not contained in $Q$. Let $H\subset \PP^4$ be a general hyperplane containing $T_pD$. We saw in the proof of Lemma \ref{aa22} that a general element of $\psi ^{-1}((D,p))$ corresponds to a general triple $(D,p,H)$. The set $H\cap Q$ is a smooth quadric surface. The family $\{(D,q)\}_{q\in Q\cap H \setminus D\cap H}$ has a flat limit in  $\mathbf{Hilb}(2,2)$ corresponding to $(D,p,H)$.

For the smoothness in (b), it is sufficient to check $h^0(N_{C|Q})=9$ and $h^1(N_{C|Q})=0$ for $C\in \Hh_1 \setminus (\Hh_1 \cap \Hh_2)$. If $C=D \cup \{p\}$ with $p \not\in D_{\mathrm{red}}$, then it is true since $N_{C|Q}\cong N_{D|Q}\oplus T_pQ$. Let $C$ be in $\Aa(D,p)$ for a conic $D\subset Q$ and $p\in D$. Then we have $\langle C \rangle \cong \PP^3$ and let $Q_2:=Q \cap \langle C \rangle$. From the following exact sequence
\begin{equation}
0\to N_{C|Q_2} \to N_{C|Q} \to \Oo_C(1)
\end{equation}
where $(N_{Q_2|Q})_{|C} \cong \Oo_C(1)$, we get that $h^0(N_{C|Q})\le 4+h^0(N_{C|Q_2})$. Assume that $Q_2$ is smooth and then we have $\Ii_{D,Q_2} \cong \Oo_{Q_2}(-1,-1)$. Tensoring the exact sequence
\begin{equation}
0\to \Ii_{C,Q_2} \to \Ii_{D,Q_2} \to \CC_p \to 0,
\end{equation}
by $\Oo_C$, we get 
\begin{equation}\label{er}
0\to \mathcal{T}or_1(\CC_p, \Oo_C) \to \Ii_{C,Q_2}\otimes \Oo_C \to \Oo_C(-1) \to \CC_p^{\oplus 2} \to 0.
\end{equation}
If we tensor the sequence (\ref{jkl}) by $\CC_p$, we get
$$0\to \mathcal{T}or_1(\Oo_C, \CC_p) \to \mathcal{T}or_1(\Oo_D, \CC_p) \to \CC_p \to \Oo_C \otimes \CC_p \to \CC_p \to 0$$
and so we get $\mathcal{T}or_1(\Oo_C, \CC_p) \cong \mathcal{T}or_1(\Oo_D, \CC_p)$. Tensoring the following exact sequence
$$0\to \Oo_{Q_2}(-1,-1) \to \Oo_{Q_2} \to \Oo_D \to 0$$
by $\CC_p$, we get $\mathcal{T}or_1(\Oo_D, \CC_p) \cong \CC_p$ and so $\mathcal{T}or_1(\Oo_C, \CC_p) \cong \CC_p$. Thus the sequence (\ref{er}) factors into the following two exact sequences
\begin{equation}\label{fg}
\begin{split}
0\to \CC_p\cong \mathcal{T}or_1(\CC_p, \Oo_C) \to \Ii_{C, Q_2} \otimes \Oo_C \to \Oo_D(-1,-1) \to 0,\\
0\to \Oo_D(-1,-1) \to \Oo_C(-1,-1) \to \Oo_C \otimes \CC_p \cong \CC_p^{\oplus 2} \to 0.
\end{split}
\end{equation}
Applying $\Ext^{\bullet}_{Q_2} (-, \Oo_C)$-functor to the first sequence of (\ref{fg}), we get
$$0\to \Hom_{Q_2}(\Oo_D(-1,-1), \Oo_C) \to \Hom_{Q_2}(\Ii_{C, Q_2} \otimes \Oo_C, \Oo_C) \to \Hom_{Q_2}(\CC_p, \Oo_C) \to \cdots.$$
Note that $\dim \Hom_{Q_2}(\CC_p, \Oo_C)=1$ and $\dim \Hom_{Q_2}(\Oo_D(-1,-1), \Oo_C)=4$. It implies that $h^0(N_{C|Q_2})=\dim \Hom_{Q_2}(\Ii_{C,Q_2}\otimes \Oo_C, \Oo_C)$ is at most $5$. Since $C$ is a limit of flat family of curves $D\cup \{q\}\subset Q_2$ with $q \not \in D$ and $h^0(N_{C|Q_2})$ is an upper semi-continuous function on $C$, so we get $h^0(N_{C|Q_2})=5$. Thus we have $h^0(N_{C|Q})\le 9$. Again by the upper semi-continuity, we get $h^0(N_{C|Q})=9$. 

If $Q_2$ is a singular quadric surface with the singular point $e\ne p$, then we may apply the exact same steps as in (a) by replacing $\Oo_{Q_2}(1,1)$ by $\Oo_{Q_2}(1)$, because we still have that (i) $\Ii_{D,Q_2} \cong \Oo_{Q_2}(-1)$, (ii) $\mathcal{T}or_1(\Oo_C, \CC_p) \cong \CC_p$ and (iii) $\dim \Hom_{Q_2}(\Oo_D(-1), \Oo_C)=4$. 
\end{proof}

\begin{lemma}\label{aaa0}
For $[D]\in \mathbf{Hilb}(2,1)$ and a point $p\in D$, there exists a unique non-trivial extension 
\begin{equation}\label{ext11}
0\to \Oo_D \to \Ff \to \CC_p \to 0.
\end{equation}
\end{lemma}

\begin{proof}
Let $D\in \mathbf{Hilb}(2,1)$ be a conic in $Q$ defined as a complete intersection of zeros of two linear forms $f_1, f_2$ with the exact sequence
\begin{equation}\label{ci}
0\to \Oo_{Q}(-2) \to \Oo_{Q}(-1)^{\oplus 2} \to \Ii_D \to 0.
\end{equation}
Applying $\Ext^{\bullet} (\CC_p, - )$-functor to (\ref{ci}) for a point $p \in D$, we get
$$0\to \Ext^2 (\CC_p, \Ii_D) \to \Ext^3 (\CC_p,\Oo_Q(-2)) \stackrel{s}{\to} \Ext^3 (\CC_p, \Oo_Q(-1))^{\oplus 2}.$$
Note that we get that $\dim \Ext^3 (\CC_p, \Oo_Q(-2))=\dim \Ext^3(\CC_p, \Oo_Q(-1))=1$ and the map $s$ is the transpose of $\Hom (\Oo_Q(2), \CC_p)^{\oplus 2} \to \Hom (\Oo_Q(1), \CC_p)$ defined by $(g_1, g_2) \mapsto f_1g_1+f_2g_2$. Since $p$ is a point on $D$, the map $s$ is a zero map and so we have $\dim \Ext ^2(\CC_p, \Ii_D)=1$. Applying $\Ext^{\bullet} (\CC_p, -)$-functor to the standard sequence for $\Oo_D$, we get $\Ext^1 (\CC_p, \Oo_D) \cong \Ext^2 (\CC_p, \Ii_D)$ and so there exists a unique non-trivial extension.
\end{proof}

\begin{proposition}\label{class}
Any semistable sheaf $\Ff$ of depth $1$ on $Q$ with Hilbert polynomial $2t+2$ is one of the following:
\begin{itemize}
\item [(i)] $\Ff \cong \Oo_{L_1} \oplus \Oo_{L_2}$ with two lines $L_1, L_2$ on $Q$, possibly $L_1=L_2$,
\item [(ii)] a non-trivial extension of $\Oo_L$ by $\Oo_L$ with $L$ a line on $Q$,
\item [(iii)] $\Ff \cong \Oo_C(p)$ for a smooth conic $C$ and $p\in C$. 
\end{itemize}
\end{proposition}

\begin{proof}
Let $C=C_{\Ff}$ be the scheme-theoretic support of $\Ff$ and then $\deg (C)$ is either $1$ or $2$. If $\deg (C)=1$, i.e. $C=L$ is a line, then $\Ff$ is a vector bundle of rank $2$ on $L$. Thus we have $\Ff \cong \Oo_L(a_1)\oplus \Oo_L(a_2)$ for $a_1\ge a_2$ such that $a_1+a_2=0$. In particular the subbundle $\Oo_L(a_1)$ destabilize $\Ff$ unless $a_1=0$ and so we have $\Ff\cong \Oo_L^{\oplus 2}$. 

Now let us assume that $\deg (C)=2$. Since $\chi(\Ff)=2$, so we have $h^0(\Ff)\ge 2$. A non-zero section of $H^0(\Ff)$ induces an injection $u: \Oo_D \hookrightarrow \Ff$ for a subcurve $D\subset C$. If $D\subsetneq C$, then $D=L$ is a line and so we get $\mathrm{coker}(u) \cong \Oo_{L'}$ with a line $L'$, because $\chi_{\mathrm{coker}(u)}(t)=t+1$. The vanishing $\Ext^1 (\Oo_{L'},  \Oo_L)=0$ implies that $\Ff \cong \Oo_L\oplus \Oo_{L'}$. In the case of $D=C$, we get $[D] \in \mathbf{Hilb}(2,c)_+$ with $c\in \{1,2\}$ due to the semistability of $\Ff$. Note that we have $\Ff \cong \Oo_C$ if $c=2$, and $\mathrm{coker}(u)=\CC_p$ for a point $p\in C$ if $c=1$.  

Conversely, the sheaf $\Ff \cong \Oo_C$ with $[C] \in \mathbf{Hilb}(2,2)_+$ is strictly semistable from the following exact sequence
$$0\to \Oo_{L_i} \to \Ff \to \Oo_{L_j} \to 0$$
with $\{i,j\}=\{1,2\}$. Assume now $c=1$ and so $[C]\in \mathbf{Hilb}(2,2)\setminus \mathbf{Hilb}(2,2)_+$. Then $\Ff$ fits into the exact sequence (\ref{ext11}) with $p\in D$ and $[D] \in \mathbf{Hilb}(2,1)$. If $D$ is a smooth conic, then any $\Ff$ fitting into (\ref{ext11}) non-trivially is stable. Assume that $D$ is a reduced conic with the singular point $o$, say $D=L_1 \cup L_2$. If $p\in L_2 \setminus L_1$, then by applying $\mathrm{Hom}^{\bullet}( -, \Oo_{L_1})$-functor to (\ref{ext11}), we have $\Hom (\Ff, \Oo_{L_1})\cong \Hom (\Oo_C , \Oo_{L_1})$. Thus we have a surjection $u : \Ff \to \Oo_{L_1}$ with $\ker (u) \cong \Oo_{L_2}$, contradicting to the semi-stability of $\Ff$. Conversely, if $p=o$, let us assume that $\Gg\subset \Ff $ is a subsheaf with slope $>1$. With no loss of generality we may assume $\Gg \cong \Oo_{L_1}(c)$ with $c\ge 1$. The composition $u' : \Gg \to \Ff \to \CC_o$ is surjective, otherwise we would have an injection $\Oo_{L_1}(c) \to \Oo_C$, absurd. Thus we have $\ker (u') \cong \Oo_{L_1}(c-1)$ and it injects into $\Oo_C$. In particular we have $c\le 0$, a contradiction.

\quad {\emph {Claim 1 }}: We have $\Ff \cong \Oo _{L_1}\oplus \Oo _{L_2}$.

\quad {\emph {Proof of Claim 1}}: We have a surjection $u: \Oo _{L_1}\oplus \Oo _{L_2} \to \CC _p$. The sheaf $\mathrm{ker}(u)$ has depth $1$ and $2t+1$ as its Hilbert polynomial, $D$ as its support. Since the map $u_\ast: H^0( \Oo _{L_1}\oplus \Oo _{L_2}) \to H^0(\CC _p)$ is surjective,
we have $h^0(\mathrm{ker}(u)) =1$ and $h^1(\mathrm{ker}(u))=0$. By symmetry we get that the only non-zero section $s$, up to a scalar, of $\mathrm{ker}(u)$ does not vanish identically on one of the components $L_i$. We get that $s$ never vanishes, i.e. $\mathrm{ker}(u) \cong \Oo _D$. By the uniqueness of non-trivial extension due to Lemma \ref{aaa0}, we get the assertion. \qed

Now assume that $D\in \Dd_1$ with $L=D_{\mathrm{red}}$ and $p\in D$. If $\Gg \cong \Oo_L(c)$ is a subsheaf of $\Ff$ with slope $>1$, then we have $c\ge 1$. The composition $u': \Gg \to \Ff \to \CC_p$ is surjective, otherwise it would give an injection $\Oo_L(c) \to \Oo_D$ and it is absurd. Then we have $\mathrm{ker} (u') \cong \Oo_L(c-1)$ and it injects into $\Oo_D$. Since $c-1\ge 0$, so it is again absurd and so $\Ff$ is semistable. 

\quad {\emph{Claim 2}}: $\Ff$ is an extension of $\Oo_L$ by $\Oo_L$. 

\quad {\emph {Proof of Claim 2}}: Since $\Ext^1_Q(\Oo_L , \Oo_L) \cong H^0(N_{L|Q})$, so its dimension is $3$. In particular the family of non-trivial extensions of $\Oo_L$ by $\Oo_L$ as an $\Oo_Q$-sheaf up to scalar forms $\PP^2$, say $\PP (L)$. For a semistable sheaf $\Oo_C$ with $C\in \Dd_2(L)$, there exists a section $s\in H^0(\Oo_C)$, inducing an extension in $\PP (L)$. Notice that the space of sections inducing the sequence (\ref{equuu1}) form a $1$-dimensional subspace, otherwise every non-zero section of $H^0(\Oo_C)$ induce the sequence (\ref{equuu1}) and it is impossible because of $1\in H^0(\Oo_C)$. Since $\Dd_2(L)$ is isomorphic to $\PP^2 \setminus \PP^1$, so there exists a $\PP^1$-family of extensions of $\Oo_L$ by $\Oo_L$, not coming from $\Oo_C$. Since every element in $\PP (L)$ is semistable, so the only remaining possibility corresponding to $\PP^1 \subset \PP (L)$ is given as an extension of $\CC_p$ by $\Oo_D$ for $D$ the unique element in $\Dd_1(L)$ and some point $p\in L$. Indeed such an extension for every $p\in L$ admits an extension in $\PP (L)$, because $\mathrm{Aut} (Q)$ is transitive on the pairs $(L, p)$. \qed
\end{proof}

\begin{remark}
For any line $L\subset Q$ the isomorphism classes $\PP (L)$ of non-trivial extension $\Ff$ of $\Oo_L$ by itself are parametrized by $\PP^2$. We have $\PP (L)=\mathfrak{F}_1 \sqcup \mathfrak{F}_2$, where $\mathfrak{F}_1$ is parametrized by $\Dd_2(L)=\PP^2 \setminus \PP^1$ and $\mathfrak{F}_2$ is parametrized by $L\cong \PP^1$, i.e. the extensions of $\CC_p$ by $\Oo_D$ with $\{D\}=\Dd_1(L)$ and $p\in L$. So the corresponding locally CM curve to $\Ff \in \mathfrak{F}_2$ is the unique $D\in \Dd_1(L)$. 

Moreover, the sheaf $\Ff\in \mathfrak{F}_2$ is not isomorphic to $\Oo_L^{\oplus 2}$. Indeed from two non-proportional surjective maps $u_i: \Ff \to \Oo _L$ for $i=1,2$, we get two non-proportional, possibly not surjective, maps $v_i: \Oo _D\to \Oo _L$. One of them cannot be surjective and so it has $\Oo _L(-c)$ with $c\ge 1$, which is impossible since $\Oo_D$ is stable by Proposition \ref{a1}. 
\end{remark}

\begin{remark}\label{ghj}
Let $C$ be an abstract ribbon with $\chi (\Oo_C) =2$ and $\Ff$ be the unique non-trivial extension of $\CC_p$ by $\Oo_D$ with $D\in \Dd_1(L)$ and $p\in L$. There is a unique morphism $u: C\to D$, the blow-up of $D$ at $p$ in the terminology of \cite{eg}; set $\Gg:= u_* \Oo_C$ and then it is an $\Oo_D$-module with an injection $\Oo_D \to \Gg$ with $\CC_p$ as its cokernel. Since $\Gg$ is torsion-free, this extension of $\CC_p$ as an $\Oo_D$-module is not trivial. Since $\Oo_D$ is a quotient of $\Oo_Q$, so $\Gg$ is also an extension as $\Oo_Q$-module and it is non-trivial, since it is torsion-free as an $\Oo_D$-module. Hence we get $\Gg\cong \Ff$. $\Ff$ is not locally free at $p$, because it is locally free outside $p$ and we know that all line bundles on $D$ have odd Euler number with the form $\Oo_D(t)$, $t\in \ZZ$. Thus the fiber of $\Ff$ at $p$ is not a $1$-dimensional vector space by Nakayama's lemma. Hence $\Ff_{|L}$ is not a line bundle on $L$ and so we get $\Ff_{|L} \cong \Oo_L \oplus \CC_p$. 
\end{remark}

\begin{corollary}\label{class+}
Any $[\Ff ]\in \mathbf{M}(2,2)$ appears in this list:
\begin{itemize}
\item [(i)] $\Ff \cong \Oo_{L_1} \oplus \Oo_{L_2}$ with two lines $L_1, L_2$ on $Q$, possibly $L_1=L_2$,
\item [(ii)] $\Ff \cong \Oo_C(p)$ for a smooth conic $C$ and $p\in C$. 
\end{itemize}
\end{corollary}

\begin{remark}
\begin{enumerate}
\item For each sheaf $\Ff \in \PP (L)$, we get the exact sequence
\begin{equation}\label{b}
0\to \Oo_L \to \Ff \to \Oo_L \to 0
\end{equation}
and so its corresponding equivalence class in $\mathbf{M}(2,2)$ coincides with $[\Oo_L^{\oplus 2}]$. 
\item In (iii) of Proposition \ref{class} with $D$ a smooth conic, we have $\Ff \cong \Oo_C(p)$ and so $\Ff$ is determined only by the choice of $D$. 
\end{enumerate}
\end{remark}

\begin{lemma}\label{123}
For a semistable sheaf $\Ff$ of depth $1$ on $Q$ with Hilbert polynomial $2t+2$, we have 
$$
\dim \Ext^1_Q(\Ff, \Ff)=\left\{
                                           \begin{array}{llll}
                                             12, & \hbox{if $\Ff\cong \Oo_L^{\oplus 2}$ for a line $L\subset Q$,}\\                                           
                                             8, & \hbox{if $C_{\Ff}\in \GG_2$,}\\
                                             6, & \hbox{otherwise.}
                                           \end{array}
                                         \right.\
$$
Recall that we have $C_{\Ff}\in \GG_2$ if and only if either 
\begin{itemize}
\item[(a)] $\Ff\cong \Oo_{L_1}\oplus \Oo_{L_2}$ for two lines $L_1,L_2$ with $|L_1 \cap L_2|=1$ or 
\item[(b)] $\Ff$ admits (\ref{ext11}) non-trivially with $D\in \Dd_1$ and $p\in D_{\mathrm{red}}$.
\end{itemize}
\end{lemma}

\begin{proof}
By \cite[Lemma 13]{CCM} we have the following exact sequence for a smooth subvariety $Y$, a coherent $\Oo_Q$-sheaf $\Ff$ and a coherent $\Oo_Y$-sheaf $\Gg$:
\begin{equation}\label{ggg}
\begin{split}
0&\to \Ext^1_Y(\Ff_{|Y},\Gg) \to \Ext^1_Q(\Ff, \Gg) \\ &\to \Hom_Y(\mathcal{T}or_1^Y(\Ff, \Oo_Y), \Gg) \to \Ext^2_Y(\Ff_{|Y}, \Gg).
\end{split}
\end{equation}
If $\Ff\cong \Gg \cong \Oo_C(p)$ for a smooth conic $C$ and a point $p\in C$, we get $\Ext^1_Q (\Ff, \Ff)\cong H^0(N_{C|Q})^{\oplus 2}$ and so its dimension is $6$. 

Let $\Ff \cong \Oo_{L_1}$ and $\Gg \cong \Oo_{L_2}$ for two lines $L_1$ and $L_2$. If $L:=L_1=L_2$, then we get $\Ext^1_Q(\Ff, \Gg) \cong H^0(N_{L|Q})$ whose dimension is $3$. If $L_1 \cap L_2=\emptyset$, then we get $\Ext^1_Q(\Ff, \Gg)=0$. If $L_1 \cap L_2=\{p\}$, then $\Ext^1_Q(\Ff, \Gg) \cong \Ext^1_{L_2}(\CC_p, \Oo_{L_2})$ and so its dimension is $1$. Using this we get for $\Ff=\Oo_{L_1}\oplus \Oo_{L_2}$ that 
$$
\dim \Ext^1_Q(\Ff, \Ff)=\left\{
                                           \begin{array}{lll}
                                             12, & \hbox{if $L_1=L_2$;}\\                                           
                                             8, & \hbox{if $|L_1 \cap L_2|=1$;}\\
                                             6, & \hbox{if $L_1 \cap L_2=\emptyset$.}
                                           \end{array}
                                         \right.\
$$  

Assume now that $\Ff \cong \Oo_C$ with $C\in \Dd_2$ and $L:=C_{\mathrm{red}}$. By Claim of Lemma \ref{a2}, we have $Q_1:=\langle C\rangle \cap Q$ is a smooth quadric surface. Assume that $C\in |\Oo_{Q_1}(2,0)|$. Apply the sequence (\ref{ggg}) to $Y=Q_1$ with $\Ff=\Gg=\Oo_C$ and then we have
\begin{equation}\label{666}
\begin{split}
0&\to \Ext^1_{Q_1} (\Oo_C, \Oo_C) \to \Ext^1_Q(\Oo_C, \Oo_C) \\ &\to \Hom_{Q_1}(\Oo_C(-1), \Oo_C)\to \Ext^2_{Q_1}(\Oo_C, \Oo_C).
\end{split}
\end{equation}
Note that the dimension of $\Hom_{Q_1}(\Oo_C(-1), \Oo_C)$ is $h^0(\Oo_C(1))=4$. Apply $\Ext^{\bullet}_{Q_1}(\Oo_C, -)$ to the standard exact sequence for $C \subset Q_1$ and then we get
\begin{align*}
0&\to \Hom_{Q_1}(\Oo_C, \Oo_C) \to \Ext^1_{Q_1}(\Oo_C, \Oo_{Q_1}(-2,0))\\
 &\to \Ext^1_{Q_1}(\Oo_C, \Oo_{Q_1}) \to \Ext^1_{Q_1}(\Oo_C, \Oo_C)\to \Ext^2_{Q_1}(\Oo_C, \Oo_{Q_1}(-2,0))\\
 &\to \Ext^2_{Q_1}(\Oo_C, \Oo_{Q_1}) \to \Ext^2_{Q_1}(\Oo_C, \Oo_C) \to 0.
\end{align*}
Note that $\Ext_{Q_1}^1(\Oo_C, \Oo_{Q_1}(-2,0)) \cong H^1(\Oo_C(0,-2))^\vee \cong H^0(\Oo_C(2,0))\cong H^0(\Oo_C)$ and so its dimension is $2$. Similarly we get that 
$$\dim \Ext^1_{Q_1}(\Oo_C, \Oo_{Q_1})=2~~\text{ and } ~~\dim \Ext^2_{Q_1}(\Oo_C, \Oo_{Q_1}(-2,0))=0.$$ 
Since the dimension of $\Hom_{Q_1}(\Oo_C ,\Oo_C)$ is at least $2$, it is indeed $2$ and so we get $\dim \Ext^1_{Q_1}(\Oo_C, \Oo_C)=2$. Since $\Ext^2(\Oo_C, \Oo_{Q_1})=0$ by the Serre duality, so we get $\Ext^2_{Q_1}(\Oo_C, \Oo_C)=0$. Hence we get $\dim \Ext^1_Q(\Oo_C, \Oo_C)=6$ by the sequence (\ref{666}). 

Let $\Ff$ be the non-trivial extension of $\CC_p$ by $\Oo_D$ with $D\in \Dd_1(L)$ and $p\in L$. Now we get the following two sequences  for $\Ii_D$: (\ref{ci}) and 
\begin{equation}\label{c}
0\to \Ii_D \to \Oo_Q \to \Oo_D \to 0,
\end{equation}
Apply $\Ext^{\bullet}_Q(\Oo_L, -)$-functor to (\ref{ci}) and then we get
\begin{equation}
0\to \Ext^1_Q(\Oo_L, \Ii_D) \to \Ext^2_Q(\Oo_L, \Oo_Q(-2)) \to \Ext^2_Q(\Oo_L, \Oo_Q(-1)^{\oplus 2}) \to \Ext^2_Q(\Oo_L, \Ii_D) \to 0
\end{equation}
since $\Ext^1_Q(\Oo_L, \Oo_Q(-1)^{\oplus 2})=\Ext^3_Q(\Oo_L, \Oo_Q(-2))=0$. Now we get $\Ext^1_Q(\Oo_L, \Oo_Q(-2))=0$ and $\Ext^2_Q(\Oo_L, \Oo_Q(-1)) \cong H^1(\Oo_L(-2))^\vee$. Thus we have 
$$
\dim \Ext^i_Q(\Oo_L, \Ii_D)=\left\{
                                           \begin{array}{ll}
                                            0, & \hbox{if $i\ne 2$,}\\                                           
                                             2, & \hbox{if $i=2$.}
                                           \end{array}
                                         \right.\
$$
Apply $\Ext^{\bullet}_Q(\Oo_L, -)$-functor to (\ref{c}) and then we get
\begin{equation}\label{e}
0\to \Ext^1_Q(\Oo_L, \Oo_D) \to \Ext^2_Q(\Oo_L, \Ii_D) \to \Ext^2_Q(\Oo_L, \Oo_Q) \to \Ext^2_Q(\Oo_L, \Oo_D) \to 0
\end{equation}
since $\Ext^1_Q(\Oo_L, \Oo_Q) \cong H^2(\Oo_L(-3))^\vee=0$. Note that $\Ext^2_Q(\Oo_L, \Oo_Q) \cong H^1(\Oo_L(-3))^\vee$ and so its dimension is $2$. We have $\Ext^2_Q(\Oo_L, \Oo_D) \cong \Ext^1_Q(\Oo_D, \Oo_L(-3))^\vee$ and it admits the following exact sequence
$$0\to \Ext^1_L(\Oo_D \otimes \Oo_L, \Oo_L(-3)) \to \Ext^1_Q(\Oo_D, \Oo_L(-3)) \to \Hom_L(\mathcal{T}or_1^Q(\Oo_D, \Oo_L), \Oo_L(-3)). $$
In particular, the dimension of $\Ext^1_Q(\Oo_D, \Oo_L(-3))$ is at least $\dim \Ext^1_L(\Oo_D \otimes \Oo_L, \Oo_L(-3))=2$ since we have $\Oo_D \otimes \Oo_L \cong \Oo_L$. Hence from (\ref{e}) we get
$$
\dim \Ext^i_Q(\Oo_L, \Oo_D)=\left\{
                                           \begin{array}{ll}
                                             2, & \hbox{if $i=1,2$,}\\                                           
                                             0, & \hbox{otherwise.}
                                            \end{array}
                                         \right.\
$$
Now we get two exact sequences (\ref{ext11}) and (\ref{b}). Apply $\Ext^{\bullet}_Q(\Oo_L, -)$-functor to (\ref{ext11}) to get 
\begin{align*}
0&\to \Ext^1_Q(\Oo_L, \Oo_D) \to \Ext^1_Q(\Oo_L, \Ff) \to \Ext^1_Q(\Oo_L, \CC_p) \\
&\to \Ext^2_Q(\Oo_L, \Oo_D) \to \Ext^2_Q(\Oo_L, \Ff)\to \Ext^2_Q(\Oo_L, \CC_p) \to 0,
\end{align*}
since $\Hom_Q(\Oo_L, \Ff) \cong \Hom_Q(\Oo_L, \CC_p)$ with dimension $1$. We also get the exact sequence
$$0\to \Ext^1_L(\Oo_L, \CC_p) \to \Ext^1_Q(\Oo_L, \CC_p) \to \Hom_L(\mathcal{T}or_1^Q(\Oo_L, \Oo_L), \CC_p)\to 0,$$
which gives $\Ext^1_Q(\Oo_L, \CC_p) \cong \Hom_L(N_{L|Q}^\vee, \CC_p)$. Thus we get $\dim \Ext^1_Q(\Oo_L, \CC_p)=2$. 

Recall that $\Ff_{|L} \cong \Oo_L\oplus \CC_p$ by Remark \ref{ghj} and so from the exact sequence
$$0\to \Ext^1_L (\Ff_{|L}, \Oo_L(-3)) \to \Ext^1_Q(\Ff, \Oo_L(-3)) \to \Hom_L(\mathcal{T}or_1^Q(\Ff, \Oo_L), \Oo_L(-3)) \to 0,$$
we get $\dim \Ext^1_Q(\Ff, \Oo_L(-3))=3+\dim \Hom_L(\mathcal{T}or_1^Q(\Ff, \Oo_L), \Oo_L(-3))$. Here we get
$$\mathcal{T}or_1^Q(\Ff, \Oo_L) \cong \mathcal{T}or_1^Q(\Ff_{|L} , \Oo_L) \cong \mathcal{T}or_1^Q(\Oo_L, \Oo_L) \cong N_{L|Q}^\vee$$
since $\mathcal{T}or_1^Q(\CC_p, \Oo_L)$ is zero. Since $\Hom_L(N_{L|Q}^\vee, \Oo_L(-3))$ is trivial and $\Ext^2_Q(\Oo_L, \Ff) \cong \Ext^1_Q(\Ff ,\Oo_L(-3))^\vee$, so we get $\dim \Ext^2_Q(\Oo_L, \Ff)=3$. Hence we have
$$
\dim \Ext^i_Q(\Oo_L, \Ff)=\left\{
                                           \begin{array}{llll}
                                             1, & \hbox{if $i=0$,}\\                                           
                                             4, & \hbox{if $i=1$,}\\
                                             3, & \hbox{if $i=2$,}\\
                                             0, & \hbox{otherwise}
                                            \end{array}
                                         \right.\
$$
Now apply $\Ext^{\bullet}_Q(-, \Ff)$-functor to (\ref{b}) and then we get
\begin{align*}
0&\to \Ext^1_Q(\Oo_L, \Ff) \to \Ext^1_Q(\Ff, \Ff) \to \Ext^1_Q(\Oo_L, \Ff)\\
&\to \Ext^2_Q(\Oo_L, \Ff) \to \Ext^2_Q(\Ff, \Ff) \to \Ext^2_Q(\Oo_L, \Ff)\to 0,
\end{align*}
since $\dim \Hom_Q(\Oo_L, \Ff)=1$ and $\dim \Hom_Q(\Ff, \Ff)=2$. Hence $\dim \Ext^1_Q(\Ff, \Ff)$ is at most $8$. Note that any curve in $\GG_1$ is a specialization of curves in $\GG_2 \setminus \GG_1$ and so $\Ff$ can be considered as a degeneration of $\Oo_{L_1}\oplus \Oo_{L_2}$ for a reducible and reduced conic $L_1 \cup L_2$. By \cite[Theorem in page 21]{Banica} in case of $N=n=1$, the global $\Ext^1$-function is upper semi-continuous. Since $\dim \Ext^1_Q(\Gg, \Gg)=8$ for $\Gg=\Oo_{L_1}\oplus \Oo_{L_2}$, so $\dim \Ext^1_Q(\Ff, \Ff)$ is at least $8$. Hence the dimension is indeed $8$. 
\end{proof}

\begin{remark}
For a sheaf $\Ff\cong \Oo_C(p)$ with a smooth conic $C$ and $p\in C$, we have $\Ext^2_Q(\Oo_C(p), \Oo_C(p)) \cong  \Ext^1_Q(\Oo_C(p), \Oo_C(p) \otimes \Oo_Q(-3))^\vee$. Applying the sequence (\ref{ggg}) to $\Ff=\Oo_C(p)$ and $\Gg=\Oo_C(p)\otimes \Oo_Q(-3)$, we get
$$0\to \Ext^1_C(\Ff, \Ff (-3))\to \Ext^1_Q(\Ff, \Ff\otimes \Oo_Q(-3)) \to H^0(N_{C|Q}(-3))\to 0.$$
Since $h^0(N_{C|Q}(-3))=0$ and $\dim \Ext^1_C(\Ff, \Ff(-3))=5$, we get $\dim \Ext^2_Q(\Ff, \Ff)=5$. 
\end{remark}

\begin{corollary}
The moduli $\mathbf{M}(2,2)$ has the two irreducible components $\mathfrak{M}_1$ and $\mathfrak{M}_2$ such that 
\begin{itemize}
\item $\mathfrak{M}_{1,\mathrm{red}}$ is parametrized by $\GG_3=\mathrm{Gr}(3,5)$,
\item $\mathbf{M}^s(2,2)\subset \mathfrak{M}_1$ consisting of stable sheaves, is isomorphic to $\GG_3\setminus \GG_2$,
\item $\mathfrak{M}_2$ consists only of strictly semi-stable sheaves $[\Oo_{L_1}\oplus \Oo_{L_2}]$ with two lines $L_1, L_2$ on $Q$,
\item $(\mathfrak{M}_1 \cap \mathfrak{M}_2)_{\mathrm{red}}$ is parametrized by $\GG_2$, consisting of the equivalence classes $[\Oo_{L_1}\oplus \Oo_{L_2}]$ with two intersecting lines $L_1, L_2\subset Q$.
\end{itemize}
\end{corollary}

\begin{lemma}
We have 
$$\mathbf{M}(2,2)^{\infty} \cong \mathbf{M}(2,2)^{\alpha} \cong \mathbf{M}(2,2)^{0+}$$
for all $\alpha>0$. 
\end{lemma}
\begin{proof}
As in Proposition \ref{a1}, let $(1,\Ff)\in \mathbf{M}^{\alpha}(2,2)$ be a strictly $\alpha$-semistable pair and so it has a subpair $(s,\Ff')$ with $\chi_{\Ff'}(t)=t+c$ such that $c+\delta \cdot \alpha=\frac{2+\alpha}{2}$. In particular we have $\Ff' \cong \Oo_{L'}(c-1)$ for a line $L'\subset Q$. If $\delta=1$, then we have $2c+\alpha=2$. But since $\Ff'$ has a non-zero section, so we have $c\ge 1$, a contradiction. If $\delta=0$, then we have $c=\frac{2+\alpha}{2}$. But the quotient pair $(1,\Ff'')$ has $\chi_{\Ff''}(t)=t+1-c$. Since $\Ff'' \cong \Oo_{L''}(-c)$ with a line $L''$, so we have $c\le 0$, a contradiction. Thus there is no wall-crossing among $\mathbf{M}^{\alpha}(2,2)$'s.
\end{proof}

\begin{proposition}
We have $\mathbf{M}^{\infty}(2,2)=\mathfrak{N}_1 \cup \mathfrak{N}_2$ with two irreducible components $\mathfrak{N}_1$ and $\mathfrak{N}_2$ of dimension $7$ and $6$ respectively such that 
\begin{itemize}
\item $\mathfrak{N}_1$ and $\mathfrak{N}_2$ are rational, 
\item $\mathbf{M}^{\infty}(2,2)$ is smooth outside $\mathfrak{N}_1 \cap \mathfrak{N}_2$, and 
\item $(\mathfrak{N}_1\cap \mathfrak{N}_2)_{\mathrm{red}}$ is of dimension $5$ and it consists of stable pairs $(1,\Ff)$ where $C_{\Ff}\in \GG_2$. 
\end{itemize} 
\end{proposition}

\begin{proof}
Let $\Lambda:=(1,\Ff)$ be a stable pair in $\mathbf{M}^{\infty}(2,2)$ and then $\Ff$ is semistable as an element of $\mathbf{M}(2,2)$, since $(1,\Ff)$ is also in $\mathbf{M}^{0+}(2,2)$. In Proposition \ref{class} we have the list of semistable sheaves and in particular we have $h^1(\Ff)=0$. Note that the tangent space at $[\Lambda]$ is isomorphic to $\Ext^1_Q(\Lambda, \Lambda)$ and the obstruction space is $\Ext^2_Q(\Lambda, \Lambda)$.  By \cite[Corollary 1.6]{He}, we get the following exact sequence
\begin{equation}\label{r12}
\begin{split}
0&\to \Hom_{Q}(\Lambda, \Lambda ) \to \Hom_Q(\Ff, \Ff) \to \Hom_{\CC}(\langle 1 \rangle , H^0(\Ff)/\langle 1 \rangle)\\
& \to \Ext^1_Q(\Lambda, \Lambda) \to \Ext^1_Q (\Ff, \Ff) \to 0.
 \end{split}
\end{equation}
and an isomorphism $\Ext^2_Q(\Lambda, \Lambda)\cong \Ext^2_Q(\Ff, \Ff)$. 

Take $\Ff\cong \Oo _L^{\oplus 2}$. Since $H^0(\Ff )= H^0(L,\Oo _L^{\oplus 2})$, so any nonzero section $s\in H^0(\Ff)$ may be seen as a non-zero map $\sigma: \Oo _L\to \Ff$. The sheaf $\sigma(\Oo _L)$ is a subsheaf of $\Ff$ obviously isomorphic to $\Oo _L$ and the pair $(s,\sigma (\Oo _L))$ shows that $(s,\Ff)$ is not $\alpha$-semistable for any $\alpha>0$.

\quad{(a1)} If $\Ff \cong \Oo_C(p)$ for a smooth conic $C$ and $p\in C$, a pair $(s,\Ff)$ for any non-zero section $s\in H^0(\Oo_C(p))$ is $\alpha$-stable. In Lemma \ref{123}, we observed that $\dim \Ext^1_Q(\Ff, \Ff)=6$ and so we get $\dim \Ext^1_Q(\Lambda, \Lambda)=7$ by (\ref{r12}). Since the family of such pairs is $7$-dimensional, so $\mathbf{M}^{\infty}(2,2)$ is smooth at $(1,\Ff)$. 

\quad{(a2)} Take $\Ff \cong \Oo_{L_1}\oplus \Oo_{L_2}$ with two skew lines $L_1, L_2 \subset Q$. As in (d), any section $s=(s_1, s_2) \in H^0(\Ff)$ with nonzero $s_1$ and $s_2$, defines $\alpha$-stable pair $(s,\Ff)$ for any $\alpha>0$. Note that $\dim\Hom_Q(\Lambda, \Lambda)=1$ and $\dim \Hom_Q(\Ff, \Ff)=2$. Since $\dim \Ext^1_Q(\Ff, \Ff)=6$ by Lemma \ref{123}, we have $\dim \Ext^1_Q(\Lambda, \Lambda)=6$. Since the family of two skew lines on $Q$ is $6$-dimensional, so $\mathbf{M}^{\infty}(2,2)$ is smooth at $(1,\Ff)$.

Let $\mathfrak{N}_1$ be the closure of stable pairs of type (a1) and $\mathfrak{N}_2$ be the closure of stable pairs of type (a2). By (a1) and (a2), we get that $\mathfrak{N}_1$ and $\mathfrak{N}_2$ are two different irreducible components of $\mathbf{M}^{\infty}(2,2)$.

\quad{(b)} Take $\Ff \cong \Oo _C$ with $C\in \Dd_2$. We have $\Ff \not \cong \Oo _L^{\oplus 2}$ as an $\Oo _Q$-sheaf, because it is not an $\Oo _L$-sheaf. In particular $\Ee$ is semistable and indecomposable. By (\ref{equuu1}) we have $h^0(\Ff )=2$, $h^1(\Ff )=0$ and there is a $1$-dimensional linear subspace $V\subset H^0(\Ff)$ associated to the inclusion $\Oo _L\to \Oo _C$
in (\ref{equuu1}). For any $s\in V\setminus \{0\}$, the inclusion $(s,\Oo_L)\to (s,\Ff)$ shows that $(s,\Ff)$ is not $\alpha$-semistable for any $\alpha>0$. Now take $s\in H^0(\Ff )\setminus V$.

\quad {\emph {Claim}}: $(s,\Ff)$ is $\alpha$-stable.

\quad {\emph {Proof of Claim}}: Assume that $(s,\Ff)$ is not $\alpha$-stable and take a proper subpair $(s',\Gg)$ with $\mu _\alpha (s',\Gg )\geq \mu _\alpha (s,\Ff )$. We have $\Gg \cong \Oo _L(a)$ for some $a\in \ZZ$. Since $s' \ne 0$, we get $a\ge 0$. Since (\ref{equuu1}) does not split as an $\Oo _Q$-sheaf, we get that the inclusion $\Gg \to \Ff$ is the one induced by (\ref{equuu1}) and hence $s'\in V$, a contradiction. \qed

Now for each $\alpha$-stable pair $\Lambda=(s,\Ff)$, we have $\dim \Hom_Q (\Lambda, \Lambda) =1$ and $\dim \Hom_Q(\Ff, \Ff)=2$. Thus we get $\dim \Ext^1_Q(\Lambda, \Lambda)=6$. Note that any two suitable sections of $\Ff$ define one isomorphism class of stable pairs $(1,\Ff)$ and so the family of such pairs is $5$-dimensional, because the dimension of $\Dd_2$ is $5$. Thus such pairs are contained in $\mathfrak{N}_2$ and $\mathfrak{N}_2$ is smooth at $(1,\Ff)$. 

\quad{(c)} Let $\Ff \cong \Oo_{L_1} \oplus \Oo_{L_2}$ with $D=L_1 \cup L_2$ a reducible conic and the singular point $p$ and then we have $h^0(\Ff)=2$. Let $s=(s_1, s_2) \in H^0(\Oo_{L_1}\oplus \Oo_{L_2})$ be a non-zero section. If $s_1=0$, then the section $(0,s_2)$ factors through $\Oo_{L_2}$ and it destabilize $(s,\Ff)$. Similarly we get $s_2 \ne 0$. If both $s_1$ and $s_2$ are nonzero, then $s$ gives an automorphism of $\Ff$ and so it does not factor through any proper subsheaf. In particular, $(s ,\Ff)$ is $\alpha$-stable for any $\alpha>0$. Note that $\dim \Hom_Q(\Ff, \Ff)=2$ and $\dim \Ext^1_Q(\Ff, \Ff)=8$. Since $\dim \Hom_Q (\Lambda, \Lambda)=1$, so we have $\dim \Ext_Q^1(\Lambda, \Lambda)=8$ by (\ref{r12}). Note that the family of this type of pairs is $5$-dimensional and it is contained in $(\mathfrak{N}_1 \cap \mathfrak{N}_2)_{\mathrm{red}}$ due to Lemma \ref{deg}.

\quad{(d)} Let $\Ff$ be the non-trivial extension of $\CC_p$ by $\Oo_D$ with $D\in \Dd_1(L)$ and $p\in L$. We get $h^0(\Ff)=2$ and there exists a unique section $s$ up to constant, inducing an extension (\ref{b}). For any section section $s' \in H^0(\Ff)\setminus \langle s \rangle$, we get an exact sequence (\ref{ext11}) and it gives a stable pair $\Lambda:=(s', \Ff)$. Note that $\Hom_Q(\Ff, \Ff)=2$ and $\dim \Ext^1_Q(\Ff, \Ff)=8$ by Lemma \ref{123}. Since $\dim \Hom_Q(\Lambda, \Lambda)=1$, so we have $\dim \Ext^1_Q(\Lambda, \Lambda)=8$. Note that the family of this type of pairs is $4$-dimensional and it is contained in $(\mathfrak{N}_1 \cap \mathfrak{N}_2)_{\mathrm{red}}$ due to Lemma \ref{deg}. 

Now $\mathfrak{N}_2$ is birational to $\mathrm{Sym}^2(\PP^3)$ and so it is a rational variety of dimension $6$. $\mathfrak{N}_1$ is birational to $\II_1$, since its general point $\Oo_C(p)$ with a smooth conic $C$ and $p\in C$ is determined by the pair $(C,p)$. The incidence variety $\II_1$ is a rational variety of dimension $7$, since $\II_1$ is isomorphic to the Grassmannian bundle $\mathrm{Gr}(2,T\PP^4(-1)_{|Q})$ and $T\PP^4(-1)_{|Q}$ is trivial on a nonempty subset of $Q$. 
\end{proof}

\begin{remark}
The forgetful map $\phi : \mathbf{M}^{\infty}(2,2) \to \mathbf{M}(2,2)$ consists of two morphisms $\phi_i : \mathfrak{N}_i \to \mathfrak{M}_i$ for $i=1,2$ that are surjective. 
\end{remark}

Recall that we have a morphism $\psi: \Hh_1 \to \mathbf{Hilb}(2,1)\times Q$ sending $[C]\in \Hh_1$ to $(D,p)$ with the exact sequence 
$$0\to \CC_p \to \Oo_C \to \Oo_D \to 0,$$
where $D$ is the maximal locally CM subscheme of $C$ with pure dimension $1$. Letting $\pi_1: \II_1 \to \mathbf{Hilb}(2,1)$ be the projection to $1^{\mathrm{st}}$-factor, we obtain the following diagram as a relation between $\mathbf{Hilb}(2,2)$ and $\mathbf{M}(2,2)$.
\begin{equation}
\xymatrix{
\mathbf{Hilb}(2,2)\supset \Hh_1&~~~~\psi^{-1}(\II_1)_{\mathrm{red}}  \ar@_{(->}[l]\ar[d]_{\psi} \\
&\II_1 \ar[d]_{\pi_1} \ar@{<-->}[r]& \mathfrak{N}_1\ar[d]^{\phi_1}\ar@^{(->}[r] & \mathbf{M}^{\infty}(2,2)\ar[d]^{\phi}\\
&\mathbf{Hilb}(2,1) \ar[r]^{\mathrm{red}} &\mathfrak{M}_1 ~~~~\ar@^{(->}[r] &\mathbf{M}(2,2)
}
\end{equation}
By Proposition \ref{aa22}, each fibre of $\psi$ is isomorphic to $\PP^1$ and fibres of $\pi_1$ over $\GG_3\setminus \GG_2$ are again isomorphic to $\PP^1$. Thus $\pi_1\circ \psi : \psi^{-1}(\II_1)_{\mathrm{red}} \to \GG_3 \setminus \GG_2$ is a $(\PP^1 \times \PP^1)$-fibration. Note that there exists a birational map $\eta : \II_1 \dashrightarrow \mathfrak{N}_1$ defined by $(D,p) \to \Oo_D(p)$ for a smooth conic $D$ and $p\in D$.

%%%%%%%%%%%%%%%%%%%%%%%%%%%%%%%%%%%

%\section{Resolution of sheaves}
%Let us denote by $\mathcal{S}$ the spinor bundle of $Q$ and define bundles $\psi_0:= \Oo_{Q}$, $\psi_1 :={\Omega_{\PP^4}^1} _{|Q}$ and $\psi_2$ fitting non-trivially into the sequence  
%$$0\to {\Omega_{\PP^4}^2 (2)} _{|Q} \to \psi_2 \to \Oo_{Q} \to 0$$
%as in \cite[Section 5]{AO}. By \cite[Theorem 5.5]{AO}, there exists a spectral sequence with $E_1$-term $E_1^{p,q}:=H^q(\Ff (p)) \otimes \psi_{-p}$ for $p>-3$ and $E_1^{-3,q}:=H^q(\Ff\otimes \mathcal{S}^\vee (-3))\otimes \mathcal{S}$, converging to $E^i=\Ff$ if $i=0$ and $0$ otherwise. 

%%%%%%%%%%%%%%%%%%%%%%%%%%%%
\providecommand{\bysame}{\leavevmode\hbox to3em{\hrulefill}\thinspace}
\providecommand{\MR}{\relax\ifhmode\unskip\space\fi MR }
\providecommand{\MRhref}[2]{%
  \href{http://www.ams.org/mathscinet-getitem?mr=#1}{#2}
}
\providecommand{\href}[2]{#2}

\end{document}